\documentclass[12pt]{amsart}

\usepackage{setspace}   

\usepackage[usenames]{color}
\usepackage{fullpage,url,mathrsfs,stmaryrd}
\usepackage{amsmath,amsthm,amssymb,mathrsfs,stmaryrd,color}
\usepackage[utf8]{inputenc}
\usepackage[T1]{fontenc}

\usepackage{relsize}
\usepackage[bbgreekl]{mathbbol}
\usepackage{amsfonts}

\DeclareSymbolFontAlphabet{\mathbb}{AMSb}
\DeclareSymbolFontAlphabet{\mathbbl}{bbold}
\newcommand{\prism}{{\mathlarger{\mathbbl{\Delta}}}}

\newcommand{\SW}{{{}^s W}}
\newcommand{\SR}{{{}^s \sR}}

\newcommand{\bu}{{\mathbf{u}}}
\newcommand{\bp}{{\mathbf{p}}}
\newcommand{\bbp}{{\mathbf{p}}}

\newcommand{\BG}{{\mathbb{G}}}

\usepackage{amssymb}
\usepackage[all]{xy}
\usepackage{mathrsfs}
\usepackage{enumerate}
\usepackage{amscd}

\usepackage{bbm}

\DeclareMathOperator{\ec}{{ec}}

\DeclareMathOperator{\Pol}{{Pol}}

\DeclareMathOperator{\pNilp}{{p-Nilp}}

\DeclareMathOperator{\perf}{{perf}}
\DeclareMathOperator{\ttop}{{top}}

\DeclareMathOperator{\HHom}{\underline{\on{Hom}}}

\newcommand{\bA}{{\mathbb A}}

\newcommand{\bN}{{\mathbb N}}

\newcommand{\cC}{{\mathcal C}}

\newcommand{\cF}{{\mathcal F}}

\newcommand{\cO}{{\mathcal O}}

\newcommand{\sR}{{\mathscr R}}

\newcommand{\sX}{{\mathscr X}}
\newcommand{\sY}{{\mathscr Y}}

\newcommand{\fL}{{\mathfrak L}}

\newcommand{\nc}{\newcommand}

\nc\wh{\widehat}

\nc\on{\operatorname}

\nc\Gr{\on{Gr}}

\nc\Fl{\on{Fl}}

\newtheorem{cor}[subsubsection]{Corollary}
\newtheorem{lem}[subsubsection]{Lemma}
\newtheorem{prop}[subsubsection]{Proposition}

\newtheorem{conj}[subsubsection]{Conjecture}

\theoremstyle{remark}
\newtheorem{rem}[subsubsection]{Remark}

\newcommand{\BF}{{\mathbb{F}}}
\newcommand{\BN}{{\mathbb{N}}}
\newcommand{\BQ}{{\mathbb{Q}}}

\newcommand{\BW}{{\mathbb{W}}}
\newcommand{\BZ}{{\mathbb{Z}}}

\DeclareMathOperator{\Lie}{{Lie}}

\DeclareMathOperator{\red}{{red}} \DeclareMathOperator{\Spf}{{Spf}}

\DeclareMathOperator{\Cone}{{Cone}}
\DeclareMathOperator{\cone}{{cone}}

\DeclareMathOperator{\BT}{{BT}}

\DeclareMathOperator{\prr}{{pr}}
\DeclareMathOperator{\Syn}{{Syn}}

\newcommand{\limfrom}{{\displaystyle\lim_{\longleftarrow}}}
\newcommand{\limto}{{\displaystyle\lim_{\longrightarrow}}}
\newcommand{\rightlim}{\mathop{\limto}}

\newcommand{\leftlim}{\mathop{\displaystyle\lim_{\longleftarrow}}}
\newcommand{\limfromn}{\leftlim\limits_{\raise3pt\hbox{$n$}}}
\newcommand{\limton}{\rightlim\limits_{\raise3pt\hbox{$n$}}}

\newcommand{\rightlimit}[1]{\mathop{\lim\limits_{\longrightarrow}}\limits%
                    _{\raise3pt\hbox{$\scriptstyle #1$}}}

\newcommand{\leftlimit}[1]{\mathop{\lim\limits_{\longleftarrow}}\limits%
                    _{\raise3pt\hbox{$\scriptstyle #1$}}}

\newcommand{\epi}{\twoheadrightarrow}
\newcommand{\iso}{\buildrel{\sim}\over{\longrightarrow}}
\newcommand{\mono}{\hookrightarrow}

\newcommand{\bcdot}{{\textstyle\cdot}}

\DeclareMathOperator{\Coker}{{Coker}}

\DeclareMathOperator{\Ext}{{Ext}}

\DeclareMathOperator{\good}{{good}} 
 
\DeclareMathOperator{\Hom}{{Hom}}

\DeclareMathOperator{\Ker}{{Ker}} \DeclareMathOperator{\id}{{id}}
\DeclareMathOperator{\im}{{Im}}

\DeclareMathOperator{\op}{{op}}

\DeclareMathOperator{\RGrpds}{{RGrpds}}

\DeclareMathOperator{\Spec}{{Spec}}

\theoremstyle{definition}

\numberwithin{equation}{section}

\newcommand{\Fr}{\operatorname{Fr}}

\newcommand{\hV}{\tilde V}

\begin{document}
\title[Ring stacks conjecturally related to the stacks $\BT_n^{G,\mu}$]{Ring stacks conjecturally related to the stacks $\BT_n^{G,\mu}$}
\author{Vladimir Drinfeld}
\address{University of Chicago, Department of Mathematics, Chicago, IL 60637}

\begin{abstract}
For $n\in\BN$, we define certain ring stacks $\SR_n$ and $\SR_n^\oplus$ using the ring space of sheared Witt vectors.
We suggest several models for the ring stacks. Motivation: there is a conjectural description of the stack of $n$-truncated Barsotti-Tate groups and its Shimurian analogs in terms of $\SR_n$ and $\SR_n^\oplus$.
\end{abstract}

\keywords{Barsotti-Tate group, p-divisible group, Witt vectors, ring stack, prismatization}

\subjclass[2020]{14L05, 14F30}


\maketitle

\tableofcontents

\section{Introduction}
Throughout this article, we fix a prime $p$.

\subsection{Conventions}   \label{ss:main conventions}
A ring $R$ is said to be $p$-nilpotent if the element $p\in R$ is nilpotent. Let $\pNilp$ denote the category of $p$-nilpotent rings.

We equip $\pNilp^{\op}$ with the fpqc topology. The word ``stack'' will mean a stack on $\pNilp^{\op}$.
The final object in the category of such stacks is denoted by $\Spf\BZ_p$; this is the functor that takes each $p$-nilpotent ring to a one-element set.

Ind-schemes and schemes over $\Spf\BZ_p$ are particular classes of stacks. The words ``scheme over $\Spf\BZ_p$'' are understood in the \emph{relative} sense (e.g., $\Spf\BZ_p$ itself is a scheme over $\Spf\BZ_p$).

$W$ will denote the functor $R\mapsto W(R)$, where $R\in\pNilp$. So $W$ is a ring scheme over $\Spf\BZ_p$. Same for $W_n$.

\subsection{Subject of the paper}
\subsubsection{The subject}
We will define and discuss certain ring stacks $\SR_n$ and $\SR_n^\oplus$, where $n\in\BN$ (more precisely, $\SR_n$ is a $(\BZ/p^n\BZ)$-algebra stack and $\SR_n^\oplus$ is a stack of $\BZ$-graded $(\BZ/p^n\BZ)$-algebras).
Our motivation for introducing these ring stacks is a conjectural relation between them and the stacks $\BT_n^{G,\mu}$ from \cite{GM}.

\subsubsection{The stacks $\BT_n^{G,\mu}$}
Let $G$ be a smooth affine group scheme over $\BZ/p^n\BZ$ and $$\mu :\BG_m\to G$$ a cocharacter satisfying the following \emph{$1$-boundedness} condition: all weights of the action of $\BG_m$ on $\Lie (G)$ are $\le 1$.
Let $\BT_n^{G,\mu}$ be the stack defined in \cite[\S 9]{GM}, so if $R\in\pNilp$ then $\BT_n^{G,\mu}(R)$ is the groupoid of $G$-bundles on $R^{\Syn}\otimes (\BZ/p^n\BZ)$ satisfying a certain condition, which depends on $\mu$.
Here $R^{\Syn}$ is the syntomification of $R$.

By Theorem D from \cite{GM}, $\BT_n^{G,\mu}$ is a smooth algebraic stack over $\Spf\BZ_p$; in other words, for every $m\in\BN$ the restriction of $\BT_n^{G,\mu}$ to the category of $\BZ/p^m\BZ$-algebras is a smooth algebraic stack over $\BZ/p^m\BZ$.
By Theorem A from \cite{GM}, if $G=GL(d)$ then $\BT_n^{G,\mu}$ identifies with the stack of $n$-truncated Barsotti-Tate groups of height $d$ and dimension $d'$, where $d'$ depends on $\mu$.

\subsubsection{Relation between $\BT_n^{G,\mu}$ and the rings stacks $\SR_n, \SR_n^\oplus$}
Conjecture D.8.4 from \cite{On the Lau} expresses $\BT_n^{G,\mu}$ in terms of the ring stacks $\SR_n$ and $\SR_n^\oplus$.
This conjecture and some variants of it were also discussed in \cite{Bonn-2025} in a rather non-technical way.

The article \cite{On the Lau} contains only a sketch of the definition of $\SR_n$ and $\SR_n^\oplus$. In this paper we give the actual definition of these ring stacks and describe several models for them.

\subsection{Sketch of the definition of $\SR_n$ and $\SR_n^\oplus$}
\subsubsection{The ideal $\hat W\subset W$}  \label{sss:2hat W}
For $R\in\pNilp$, let $\hat W(R)$ be the set of all $x\in W(R)$ such that all components of the Witt vector $x$ are nilpotent and all but finitely many of them are zero.
Then $\hat W$ is an ind-subscheme of $W$; moreover, $\hat W$ is an ideal in $W$ preserved by the operators $F,V:W\to W$. For $n\in\BN$ we set $$\hat W^{(F^n)}:=\Ker (F^n:\hat W\to\hat W), \quad W^{(F^n)}:=\Ker (F^n:W\to\hat W).$$

\subsubsection{$\SR_n$ via the ring space $\SW$}
Let $W^{\perf}$ be the projective limit of the diagram 
\[
\ldots \overset{F}\longrightarrow W\overset{F}\longrightarrow W.
\]
Let
\begin{equation} \label{e:4check W directly}
\SW=W^{\perf}/\leftlimit{n} \hat W^{(F^n)}=\leftlimit{n} (W/\hat W^{(F^n)}), 
\end{equation}
where the transition maps in each of the limits equal $F$ and the quotients are understood in the sense of fpqc sheaves on $\pNilp^{\op}$.
Thus $\SW$ is a ring space (by which we mean an fpqc sheaf of rings on $\pNilp^{\op}$); it is called the \emph{ring of sheared Witt vectors}\footnote{The name is due to the relation between $\SW$ and the theory of sheared prismatization from \cite{Akhil1,Akhil2} and~\cite{Sheared}. On the other hand, $\SW$ can be regared as a ``decompletion'' of $W$, see \S\ref{sss:Zink,Lau}(ii) and the end of \S\ref{sss:not V-complete}.}. Note that $W$ is a quotient of $\SW$: indeed, the map $W^{\perf}\to W$ is surjective, and its kernel is 
$\leftlimit{n} W^{(F^n)}\supset\leftlimit{n} \hat W^{(F^n)}$.

Now define the ring stack $\SR_n$ by
\begin{equation}  \label{e:2hat sR_n}
\SR_n:=\Cone (\SW\overset{p^n}\longrightarrow\SW).
\end{equation}

\subsubsection{$\SW^\oplus$ and $\SR_n^\oplus$}  \label{sss:how we get the graded rings}
The homomorphism $F:W\to W$ induces a homomorphism $$F:\SW\to\SW.$$
One also has an important additive homomorphism $\hV :\SW\to\SW$; it is defined 
using the operator $V:W\to W$ \emph{in a nontrivial way} (see \S\ref{ss:2hat V}, in which we follow \cite{Akhil2,BMVZ}).
Applying to $(\SW,F,\hV )$ a certain general algebraic construction (which we call the \emph{Lau equivalence}),
one gets a $\BZ$-graded ring space $\SW^\oplus$, see \S\ref{sss:check W oplus}.
Finally, one sets 
\begin{equation}  \label{e:2hat sR_n oplus}
\SR_n^\oplus:=\Cone (\SW^\oplus\overset{p^n}\longrightarrow\SW^\oplus).
\end{equation}

\subsection{Remarks on $\SW$ and $\hV$}
\subsubsection{}
The definition of $\hV :\SW\to\SW$ is not obvious from \eqref{e:4check W directly} because in mixed characteristic we have $FV\ne VF$. On the other hand, in characteristic $p$ one has $FV= VF$, so the operator
$\hV :\SW_{\BF_p}\to\SW_{\BF_p}$ is clear from \eqref{e:4check W directly} (here $\SW_{\BF_p}:=\SW\times\Spec\BF_p$).

\subsubsection{}
To define $\hV :\SW\to\SW$, it is convenient to replace \eqref{e:4check W directly} by the equivalent formula~\eqref{e:Definition of check W}.

\subsubsection{Some history}  \label{sss:Zink,Lau}
(i) I suggested the definition of $\SW$ while thinking about \cite{Vadik's talk,Sheared} and about the stacks $\BT_n^{G,\mu}$. Simultaneously and independently, E.~Lau introduced $\SW (R)$ in the case where $R$ is a semiperfect $\BF_p$-algebra; in this case he defined $\SW (R)$ to be the right-hand side of formula~\eqref{e:SW(semiperfect)} from our Appendix~\ref{s:SW (R)}.

(ii) In \cite{Lau14} E.~Lau defined a ring $\BW (R)$ for a certain class of $p$-nilpotent rings $R$, which he calls admissible (see \S\ref{sss:admissible rings} of our Appendix~\ref{s:SW (R)}); for a smaller class it had been defined in 2001 by Th.~Zink  \cite{Zink-short}.
If $R$ is admissible then $\BW (R)=\SW (R)$ (see Proposition~\ref{p:admissible rings} of Appendix~\ref{s:SW (R)}). For admissible $R$, the operator $\hV :\BW (R)\to \BW (R)$ was defined in \cite{Zink-short} assuming that $p>2$; this assumption was removed in \cite{Lau14}.

\subsection{Models for $\SR_n$ and $\SR_n^\oplus$}
By a \emph{model} for a ring stack we mean its realization as a $\Cone$ of a quasi-ideal. The models for $\SR_n$ and $\SR_n^\oplus$ provided by \eqref{e:2hat sR_n} and \eqref{e:2hat sR_n oplus} are far from being economic.
However, there are more economic models.

\subsubsection{The situation over $\BF_p$}    \label{sss:situation over F_p}
Let $W_{\BF_p}:=W\times\Spec\BF_p$, $\hat W_{\BF_p}:=\hat W\times\Spec\BF_p$, etc. It turns out that
\begin{equation}  \label{e:reduction of sheared sR_n economically}
\SR_{n,\BF_p}=\Cone (\hat W^{(F^n)}_{\BF_p}\to W_{n,\BF_p}),
\end{equation} 
\begin{equation}  \label{e:reduction of sheared sR_n oplus economically}
\SR_{n,\BF_p}^\oplus=\Cone ((\hat W_{\BF_p}^{(F^n)})^\oplus\to W_{n,\BF_p}^\oplus), 
\end{equation} 
where $(\hat W_{\BF_p}^{(F^n)})^\oplus$ is obtained by applying the Lau equivalence to the triple
  $$(\hat W_{\BF_p}^{(F^n)}, \; F:\hat W_{\BF_p}^{(F^n)}\to \hat W_{\BF_p}^{(F^n)}, \; V:\hat W_{\BF_p}^{(F^n)}\to \hat W_{\BF_p}^{(F^n)})$$
  and $W_{n,\BF_p}^\oplus$ is obtained similarly from the triple
  $$( W_{n,\BF_p}, \; F: W_{n,\BF_p}\to  W_{n,\BF_p}, \; V: W_{n,\BF_p}\to  W_{n,\BF_p}).$$
 The reason why $V:\hat W_{\BF_p}^{(F^n)}\to \hat W_{\BF_p}^{(F^n)}$ and $F: W_{n,\BF_p}\to  W_{n,\BF_p}$ are defined is that in characteristic $p$ one has $FV=VF$.
 
 More details regarding \eqref{e:reduction of sheared sR_n economically}-\eqref{e:reduction of sheared sR_n oplus economically} can be found in \S\ref{sss:I_{n,m}/I'_{n,m} in some cases}(ii) and \S\ref{ss:applying Lau equivalence to C_{n,F_p}}.

\subsubsection{Mixed characteristic}
In mixed characteristic the situation is more complicated.
It turns out that if $p>2$ then similarly to \eqref{e:reduction of sheared sR_n economically}, one has
\begin{equation}  \label{e:sheared sR_n economically}
\SR_n=\Cone (\hat W^{(F^n)}\to W_{n}),
\end{equation} 
and if $p=2$ then $\SR_n$ has a slightly more complicated realization as a quotient of $W_n$.

However, even for $p>2$ the model \eqref{e:sheared sR_n economically} does not exhibit the operator $\hV:\SR_n\to\SR_n$. Related fact:
in mixed characteristic there is no direct analog of \eqref{e:reduction of sheared sR_n oplus economically}.

\subsubsection{Other models}
(i) $\SR_n$ has some models which are more economic than \eqref{e:2hat sR_n} but less economic than \eqref{e:sheared sR_n economically}, see \S\ref{ss:A more economic model}, \S\ref{s:hat sR_n as a quotient of W},  and \S\ref{sss:economic models for hat sR_n}.
In \S\ref{s:tilde A_n} we describe a variant of the model for $\SR_n$ from \S\ref{ss:A more economic model}, which is self-dual up to a cohomological shift, just as the model \eqref{e:sheared sR_n economically}.

(ii) The models for $\SR_n$  from \S\ref{ss:A more economic model} and \S\ref{s:tilde A_n} exhibit both $F$ and $\hV$, so they can easily be used to construct models for $\SR_n^\oplus$, see \S\ref{ss:economic model for hatsR_n oplus} and the end of \S\ref{ss:tilde A_n subject}. 
The economic models for $\SR_n$ from \S\ref{sss:economic models for hat sR_n} do not exhibit $F$ and $\hV$; still, one can use them to construct models for $\SR_n^\oplus$ (see \S\ref{ss:2economic models}).

\subsubsection{The limit $n=\infty$}
One has
\begin{equation}  \label{e:check W as limit}
\limfromn\SR_n=\SW.
\end{equation} 
So one could consider the ring stacks $\SR_n$ as primary objects and then define $\SW$ by \eqref{e:check W as limit}.

\subsection{Digression on (sheared) $n$-prismatization}
\subsubsection{The ring stacks $\sR_n$}
Similarly to \eqref{e:2hat sR_n}, let 
\begin{equation}  \label{e:sR_n}
\sR_n:=\Cone (W\overset{p^n}\longrightarrow W).
\end{equation}
One can also set $\sR_n^\oplus:=\Cone (W^\oplus\overset{p^n}\longrightarrow W^\oplus )$, where $W^\oplus$ is as in \S\ref{sss:Witt frame}(i).

\subsubsection{Recollections on prismatization}   \label{sss:prismatization}
If $X$ is an $\BF_p$-scheme then its prismatization $X^\prism$ is the functor
\begin{equation}  \label{e:crystallization}
\pNilp\to\mbox{Groupoids}, \quad A\mapsto X(\sR_1 (A)),
\end{equation}
where as before, $\sR_1:=\Cone (W\overset{p}\longrightarrow W)$; the expression $X(\sR_1 (A))$ makes sense because  $\sR_1 (A)$ is an animated $\BF_p$-algebra.
In particular, $(\bA^1_{\BF_p})^\prism =\sR_1$.

The definition of $X^\prism$ for any $p$-adic formal scheme $X$ can be found in \cite{Bh, Prismatization}. The idea is to deform the ring stack $\sR_1=\Cone (W\overset{p}\longrightarrow W)$ by replacing $p$ with $\xi$, where $\xi$ is a primitive Witt vector of degree $1$, which matters only up to multiplication by $W^\times$.

\subsubsection{Sheared prismatization}  \label{sss:sheared prismatization}
The functor of sheared prismatization,  denoted by $X\mapsto X^{\hat\prism}$  is defined in \cite{Sheared}; see also \cite{Akhil1, Akhil2}. The idea is to use $\SR_1$ instead of $\sR_1$ and to require $\xi$ to be a ``strictly primitive'' Witt vector of degree $1$.
One has $(\bA^1_{\BF_p})^{\hat\prism} =\SR_1$.

\subsubsection{(Sheared) $n$-prismatization}
For any $n\in\BN$, one could define the functor of $n$-prismati\-zation $X\mapsto X^{\prism_n}$ and its sheared version $X\mapsto X^{\hat{\prism}_n}$ so that for $n=1$ one gets the functors from \S\ref{sss:prismatization}-\ref{sss:sheared prismatization}.
To do this, replace $\sR_1$, $\SR_1$, $\BF_p$ by $\sR_n$, $\SR_n$, $\BZ/p^n\BZ$, and require $\xi$ to be (strictly) primitive of degree $n$.

To formulate Conjecture D.8.4 from \cite{On the Lau}, one only needs $X^{\hat{\prism}_n}$ in the case where $X$ is a scheme over $\BZ/p^n\BZ$. In this case we usually write $X(\SR_n )$ instead of $X^{\hat{\prism}_n}$; this is the functor
\[
\pNilp\to\mbox{Groupoids}, \quad A\mapsto X(\SR_n (A)).
\]

Let us note that if $X$ is a scheme over $\BZ/p^{n-1}\BZ$ then $X^{\prism_n}=X^{\hat{\prism}_n}=\emptyset$ (it suffices to check this if $X=\Spec\BZ/p^{n-1}\BZ$, which is easy). 

\subsection{Relation between this paper and \cite{Sheared, BMVZ}} 
(i) Most of this paper\footnote{Possible exceptions: \S\ref{ss:A simple computation}-\ref{ss:autoduality conjecture}, \S\ref{s:check W oplus}, \S\ref{s:tilde A_n}, \S\ref{s:A class of models}, a part of Appendix~\S\ref{s:SW (R)}.} is either contained in  \cite{Sheared, BMVZ} or is a straightforward modification of the material from \cite{Sheared, BMVZ} obtained by replacing the Witt vector $\xi$ from \S\ref{sss:sheared prismatization} by $p^n$. But unlike \cite{Sheared, BMVZ}, we define and discuss the graded ring space $\SW^\oplus$ and the related stacks $\SR_n^\oplus$.  

(ii) The goal of \cite{BMVZ} is to prepare the ground for a detailed theory of $\SW$ (including cohomology with coefficients in $\SW$).
  In this paper we give a relatively self-contained exposition of \emph{the more elementary part} of this theory.
  
(iii) Let $\pNilp_{{\good}}$ be the full subcategory of rings $R\in\pNilp$ such that the quotient of $R$ by its nilradical is perfect; it is known that $\pNilp_{{\good}}^{\op}$ is a \emph{base} for the fpqc topology on $\pNilp^{\op}$ (see \S\ref{sss:base of fpqc topology}).
Unlike this paper, the authors of \cite{BMVZ} prefer to consider $\SW$ as a sheaf on $\pNilp_{{\good}}^{\op}$ rather than on $\pNilp^{\op}$. Reason: from their point of view, if $R$ is not in $\pNilp_{{\good}}$ then the natural object is
$R\Gamma (\Spec R,\SW)$ rather than $\SW (R)=H^0 (\Spec R,\SW)$. The authors of \cite{BMVZ} have proved that if $R\in\pNilp_{{\good}}$ then $H^i(\Spec R, \SW )=0$ for $i>0$.

\subsection{Organization}
In \S\ref{s:Q}-\ref{s:check W} we recall the material from \cite{BMVZ, Akhil1,Akhil2} about the ring spaces $Q:=W/\hat W$ and $\SW$. We also formulate Conjecture~\ref{conj:autoduality}.

In \S\ref{s:check W oplus} we discuss the \emph{Lau equivalence} and use it to define the ring space $\SW^\oplus$.

In \S\ref{s:2the ring stacks} we define the ring stacks $\SR_n$ and $\SR_n^\oplus$. Following \cite{Sheared, Akhil2}, we construct a model for each of them, which is more economic than \eqref{e:2hat sR_n} and \eqref{e:2hat sR_n oplus}.

In \S\ref{s:tilde A_n} we slightly modify the models from \S\ref{s:2the ring stacks}. The new model for $\SR_n$ is self-dual up to a cohomological shift, just as the model \eqref{e:sheared sR_n economically}.

Following \cite{Sheared,Akhil2}, we describe in \S\ref{s:hat sR_n as a quotient of W}  a model for $\SR_n$, which represents $\SR_n$ as a quotient of $W$. 

In \S\ref{s:economic models for hat sR_n} we represent $\SR_n$ as a quotient of $W_n$; here we follow \cite{Sheared,Vadik's talk}.

In \S\ref{s:A class of models} we discuss a class of models for $\SR_n^\oplus$. 

In Appendix~\ref{s:SW (R)} we describe $\SW (R)$ assuming that the $\BF_p$-alegbra $A:=R/pR$ is \emph{weakly semiperfect}, i.e., $\Fr^n(A)=\Fr^{n+1}(A)$ for some $n\ge 0$.

In Appendix~\ref{s:derived completeness} we recall the notion of derived $p$-completeness.

\subsection{Acknowledgements}
I thank B.~Bhatt, E.~Lau, A.~Mathew, V.~Vologodsky, and M.~Zhang for explaining their works, sharing their drafts, and numerous discussions.
 
The author's work on this project was partially supported by NSF grant DMS-2001425.

\medskip

\section{The ring space $Q$} \label{s:Q}
In this section we recall some material from \cite{BMVZ,Akhil1}. As before, we use the notation and conventions of \S\ref{ss:main conventions}.

\subsection{The ideal $\hat W\subset W$}  \label{ss:hat W}
\subsubsection{Definition of $\hat W$}    \label{sss:hat W}
For $R\in\pNilp$, let $\hat W(R)$ be the set of all $x\in W(R)$ such that all components of the Witt vector $x$ are nilpotent and all but finitely many of them are zero.
Then $\hat W(R)$ is an ideal in $W(R)$ preserved by $F$ and $V$. Moreover, the preimage of $\hat W(R)$ with respect to $V:W(R)\to W(R)$ is \emph{equal} to $\hat W(R)$, so
\begin{equation} \label{e:V is injective}
\Ker (V:W(R)/\hat W(R)\to W(R)/\hat W(R))=0
\end{equation}

Recall that $W$ is an affine scheme over $\Spf\BZ_p$. Clearly, $\hat W$ is an ind-subscheme of $W$ which is ind-finite over $\Spf\BZ_p$.

\subsubsection{Surjectivity of $F$}     \label{sss:Surjectivity of F}
It is known that $F:W\to W$ is surjective as a morphism of fpqc sheaves (e.g., see \cite[\S 3.4]{Prismatization}). The same is true for 
$F:\hat W\to\hat W$; moreover,  $F:\hat W\to\hat W$ is surjective in the fppf sense, see \cite[Prop. 2.30]{BMVZ}. 

\subsubsection{$\hat W$-torsors and seminormality}      \label{sss:hat W-torsors}
According to \cite{Sw}, a ring $A$ is said to be seminormal if it is reduced\footnote{In fact, reduceness follows from the other condition, see \cite[Tag 0EUK]{Sta}.} and every homomorphism $f:\BZ[x^2,x^3]\to A$ extends to a homomorphism $\tilde f:\BZ[x]\to A$. Any perfect $\BF_p$-algebra is seminormal: in this case $\tilde f$ can be defined
by setting $\tilde f(x)=f(x^p)^{1/p}$.

The quotient of $R\in\pNilp$ by its nilradical is denoted by $R_{\red}$; note that $R_{\red}$ is an $\BF_p$-algebra.
It is proved in \cite[\S 3.1]{BMVZ} that 

(i) if $R\in\pNilp$ is such that $R_{\red}$ is seminormal then every $\hat W$-torsor on the fpqc site of $\Spec R$ is trivial (in particular, this is true if $R_{\red}$ is perfect);

(ii) for \emph{any} $R\in\pNilp$, every $\hat W$-torsor on the fpqc site of $\Spec R$ is fppf-locally trivial.

\subsubsection{Remark}      \label{sss:base of fpqc topology}
The class of $p$-nilpotent rings $R$ such that $R_{\red}$ is perfect plays a central role in \cite{BMVZ}. Such rings form a \emph{base} of the fpqc topology on $\pNilp^{\op}$; 
in other words, for every $R\in\pNilp$ there exists a faithfully flat $R$-algebra $R'$ such that $R'_{\red}$ is perfect. Moreover, there exists a faithfully flat $R$-algebra $R'$ such that $R'/pR'$ is semiperfect (this is stronger than perfectness of $R_{\red}$).
The proof is reduced to the case where $R$ is the ring of polynomials over $\BZ/p^n\BZ$ in the variables $x_i$, $i\in I$;
in this case take $R'$ to be the inductive limit of the rings $R'_m$, where $R'_m$ is the ring of polynomials over $\BZ/p^n\BZ$ in the variables $x_i^{p^{-m}}$.

\subsubsection{$\hat W (R)$, where $R$ is $p$-complete} \label{sss:hat W of p-complete}
An abelian group is said to be \emph{$p$-complete} if it is complete and separated with respect to the $p$-adic topology.
For a $p$-complete ring $R$, we define\footnote{In practice, we will apply this definition to $p$-complete rings with bounded $p^\infty$-torsion.} $\hat W (R)$ to be the projective limit of $\hat W (R/p^nR)$.
If $R$ is $p$-complete then the projective limit of $W (R/p^nR)$ equals $W(R)$, so $\hat W (R)\subset W(R)$.

\subsection{The ring space $Q$}  \label{ss:Q}
By a ring space we mean an fpqc sheaf of rings on $\pNilp^{\op}$.

\subsubsection{Definition}      \label{sss:definition of Q}
Following \cite{BMVZ}, we set $Q:=W/\hat W$ (quotient in the sense of fpqc sheaves or equivalently, in the sense of fppf sheaves; the equivalence follows from \S\ref{sss:hat W-torsors}(ii)). 

If $R\in\pNilp$ is such that $R_{\red}$ is seminormal (e.g., perfect) then $Q(R)=W (R)/\hat W (R)$ by~\S\ref{sss:hat W-torsors}.

\subsubsection{Pieces of structure on $Q$}   \label{sss:structure on Q}
It is clear that $Q$ is a ring space\footnote{Moreover, as explained in \cite{BMVZ}, $Q$ is a $\delta$-ring space, and there is a certain compatibility between $\delta$ and~$V$ (see \S\ref{ss:delta-ring structure} below).}. Moreover, the maps $F,V:W\to W$ induce maps $F,V:Q\to Q$.
By \eqref{e:V is injective} and \S\ref{sss:Surjectivity of F}, one has
\begin{equation} \label{e:Ker V=0}
\Ker (Q\overset{V}\longrightarrow Q)=0,
\end{equation}
\begin{equation} \label{e:Coker F=0}
\Coker (Q\overset{F}\longrightarrow Q)=0
\end{equation}

\subsection{The operator $\hV :Q\to Q$}  \label{ss:hat V}
\subsubsection{The story in a few words}
It turns out that if $p>2$ then the maps $F,V:Q\to Q$ satisfy the relation $FV=VF=p$. If $p=2$ this is false, but there is a better map\footnote{In a slightly different context, this map was introduced by E.~Lau \cite{Lau14}; then it was rediscovered in~\cite{BMVZ}.} $\hV :Q\to Q$ such that
$F\hV=\hV F=p$. The reader may prefer to disregard the case $p=2$ and assume that $\hV=V$.

\subsubsection{The element $\bar\bu$}   \label{sss:bar bu}
It is easy to see that there is a unique $\bar\bu\in W(\BZ_p)/\hat W(\BZ_p)$ such that 
\begin{equation}  \label{e:definition of bar bu}
V(\bar\bu )=p; 
\end{equation}
moreover, $\bar\bu$ is invertible. According to \cite[Prop.~2.27-2.28]{BMVZ}, 
\begin{equation}  \label{e:when bar bu=1}
\bar\bu =1 \Leftrightarrow  p>2.
\end{equation}
Applying $F$ to \eqref{e:definition of bar bu}, we get
\begin{equation}   \label{e:property of bar bu}
p\bar\bu =p.
\end{equation}

\subsubsection{The operator $\hV :Q\to Q$}   \label{sss:hV}
Define $\hV :Q\to Q$ by $\hV:=V\circ\bar\bu$; in other words, for $R\in\pNilp$ and $x\in Q(R)$, we have $\hV (x):=V(ux)$.
By \eqref{e:definition of bar bu}-\eqref{e:property of bar bu}, we have
\begin{equation}    \label{e:F times hV}
F\hV=\hV F=p.
\end{equation}
Of course, $\hV$ is additive and satisfies the usual identity
 \begin{equation} \label{e:usual identity}
 \hV (x)y=\hV (xF(y)).
 \end{equation}
By \eqref{e:Ker V=0}, we have
\begin{equation} \label{e:Ker hV=0}
\Ker (Q\overset{\hV}\longrightarrow Q)=0.
\end{equation}

\begin{prop}     \label{p:Q/hat V (Q)}
(i) The map 
\begin{equation}   \label{e:Q/hat V (Q)}
F:Q/\hV (Q)\to Q/\hV (Q)
\end{equation}
is an isomorphism.

(ii) Moreover, $Q/\hV (Q)$ is a sheaf of perfect $\BF_p$-algebras, and \eqref{e:Q/hat V (Q)} is its Frobenius endomorphism.
\end{prop}

The proof given below is taken from  \cite{Akhil1} and \cite[\S 3.2]{BMVZ}. 
Moreover, it is proved in \cite{Akhil1} that for any $R\in\pNilp$ the ring $(Q/\hV (Q))(R)$ is the colimit perfection of $R_{\red}$ (i.e., the colimit of the diagram
$R_{\red}\overset{F}\longrightarrow R_{\red}\overset{F}\longrightarrow \ldots$).

\begin{proof}[Proof of Proposition~\ref{p:Q/hat V (Q)}]
$Q:=W/\hat W$, so $Q/\hV (Q)=W/(V(W)+\hat W)$. Therefore $Q/\hV (Q)$ is the fpqc-sheafification of the presheaf
\[
R\mapsto R_{\red}, \quad R\in\pNilp .
\]
Moreover, the map \eqref{e:Q/hat V (Q)} comes from $\Fr :R_{\red}\to R_{\red}$. This proves (ii) and injectivity of~\eqref{e:Q/hat V (Q)}. Surjectivity of \eqref{e:Q/hat V (Q)} follows from \eqref{e:Coker F=0}.
\end{proof}

\begin{cor}     \label{c:acyclicity of bicomplex}
The complex corresponding to the bicomplex
\[
\xymatrix{
Q\ar[r]^\hV \ar[d]_F & Q\ar[d]^F\\
Q\ar[r]^\hV & Q
}
\]
is acyclic.
\end{cor}

\begin{proof}
Follows from Proposition~\ref{p:Q/hat V (Q)}(i) and formula \eqref{e:Ker hV=0}. 
\end{proof}

For $i\in\BN$ let $Q^{(F^i)}:=\Ker (Q\overset{F^i}\longrightarrow Q)$.

\begin{cor}     \label{c:2acyclicity of bicomplex}
For every $i\in\bN$ the map $\hV :Q^{(F^i)}\to Q^{(F^i)}$ is an isomorphism.
\end{cor}

\begin{proof}
Follows from Corollary~\ref{c:acyclicity of bicomplex} and formula \eqref{e:Coker F=0}. 
\end{proof}

\section{The ring space $\SW$}   \label{s:check W}
In this section we recall some material from \cite{BMVZ,Akhil2} and formulate Conjecture~\ref{conj:autoduality}.

\subsection{Definition of $\SW$}  \label{ss:Definition of check W}
\subsubsection{Definition}   \label{sss:Definition of check W}
Just as in \S\ref{s:Q}, let $Q:=W/\hat W$. Let
\begin{equation} \label{e:Definition of check W}
\SW:=W\times_Q Q^{\perf},
\end{equation}
where $Q^{\perf}$ is the projective limit of the diagram 
\begin{equation}  \label{e:Qperf as a projective system}
\ldots \overset{F}\longrightarrow Q\overset{F}\longrightarrow Q.
\end{equation}
The ring space $\SW$ is called the \emph{ring of sheared Witt vectors}.

If $R\in\pNilp$ is such that $R_{\red}$ is perfect then $Q(R)=W (R)/\hat W (R)$ by \S\ref{sss:definition of Q}, so $\SW (R)$ is rather explicit.
In Appendix~\ref{s:SW (R)} we give an even more explicit description of $\SW (R)$ if $R$ satisfies a certain condition which is stronger than perfectness of  $R_{\red}$.

\subsubsection{Reformulation}
Let $W^{\perf}$ be the projective limit of the diagram 
\[
\ldots \overset{F}\longrightarrow W\overset{F}\longrightarrow W.
\]
The canonical homomorphisms $W^{\perf}\to W$ and $W^{\perf}\to Q^{\perf}$ define a homomorphism
\begin{equation} \label{e:W perf to check W}
W^{\perf}\to\SW.
\end{equation}
Surjectivity of $F:\hat W\to\hat W$ (see \S\ref{sss:Surjectivity of F}) implies that \eqref{e:W perf to check W} is surjective as a map of fpqc sheaves, so
\begin{equation} \label{e:3check W directly}
\SW=W^{\perf}/\leftlimit{n} \hat W^{(F^n)}=\leftlimit{n} (W/\hat W^{(F^n)}), 
\end{equation}
where $\hat W^{(F^n)}:=\Ker (F^n:\hat W\to\hat W)$, the transition maps in each of the limits equal $F$, and the quotients are understood in the fpqc sense.
Note that $W^{\perf}$ is an affine scheme over $\Spf\BZ_p$ and $\leftlimit{n} \hat W^{(F^n)}$ is an ind-subscheme of $W^{\perf}$ which is also an ideal.

\subsubsection{Pieces of structure on $\SW$}   \label{sss:structure on check W}
Recall that a ring space is an fpqc sheaf of rings on $\pNilp^{\op}$.
Both \eqref{e:Definition of check W} and \eqref{e:3check W directly} exhibit $\SW$ as a ring space  equipped with an endomorphism $F$.
There is also an important additive map $\hV:\SW\to\SW$,
see \S\ref{ss:2hat V} below. Moreover, $\SW$ is a $\delta$-ring space, see \S\ref{ss:delta-ring structure} for more details.

\subsubsection{Some exact sequences}  \label{sss:Some exact sequences}
The map $W\to Q$ is clearly surjective. The map $Q^{\perf}\to Q$ is surjective by \eqref{e:Coker F=0}. So we get exact sequences of fpqc sheaves
\begin{equation} \label{e:Akhil's sequence1}
0\to\hat W\to\SW\to Q^{\perf}\to 0,
\end{equation}
\begin{equation} \label{e:Akhil's sequence2}
0\to T_F(Q)\to\SW\to W\to 0.
\end{equation}
Here $T_F(Q)$ is the ``$F$-adic Tate module'' of $Q$, i.e., 
\begin{equation}  \label{e:T_F(Q)}
T_F(Q):=\limfrom (\ldots \overset{F}\longrightarrow Q^{(F^2)}\overset{F}\longrightarrow Q^{(F)}), \quad\mbox{where } Q^{(F^n)}:=\Ker (F^n:Q\to Q).
\end{equation}

We also have a rather tautological exact sequence
\begin{equation}    \label{e:Akhil's sequence3}
0\to \hat W^{\perf}\to W\times_Q W^{\perf}\to\SW\to 0, \quad \hat W^{\perf}:=\limfrom (\ldots \overset{F}\longrightarrow \hat W\overset{F}\longrightarrow \hat W),
\end{equation}
where the map $\hat W^{\perf}\to W\times_Q W^{\perf}$ is given by $x\mapsto (0,x)$ and the map 
$$W\times_Q W^{\perf}\to W\times_Q Q^{\perf}=:\SW$$
comes from the canonical map $W^{\perf}\to Q^{\perf}$. In Appendix~\ref{s:SW (R)} we will use \eqref{e:Akhil's sequence3}
to give an explicit description of $\SW (R)$ for a certain class of rings $R\in\pNilp$.

Note that the epimorphism $W\times_Q W^{\perf}\epi W^{\perf}$ has a canonical splitting
\[
W^{\perf}\mono W\times_Q W^{\perf}, \quad x\mapsto (\pi (x),x),
\]
where $\pi:W^{\perf}\to W$ is the canonical map. So $W\times_Q W^{\perf}$ identifies with the semidirect product\footnote{Given a unital ring $A$ and a (non-unital) $A$-algebra, one defines a unital ring $A\ltimes B$; its additive group is $A\oplus B$, and the product of $(a,b)$ and $(a',b')$ is $(aa',ab'+ba'+bb')$.} $W^{\perf}\ltimes\hat W$, where $W^{\perf}$ acts on $\hat W$ via $\pi$.
After this identification, \eqref{e:Akhil's sequence3} becomes the exact sequence
\[
0\to \hat W^{\perf}\overset{(1,-1)}\longrightarrow W^{\perf}\ltimes\hat W\to\SW\to 0, 
\]
where the restriction of the map $W^{\perf}\ltimes\hat W\to\SW$  to $W^{\perf}$ is the epimorphism \eqref{e:W perf to check W} and the restriction to $\hat W$ is the embedding $\hat W\mono\SW$ from \eqref{e:Akhil's sequence1}.

\subsubsection{The ring $\SW (\BZ_p)$}   \label{sss:SW(Z_p)}
Similarly to \S\ref{sss:hat W of p-complete}, we define $\SW (\BZ_p)$ to be the projective limit of the rings $\SW (\BZ/p^n\BZ)$. The ring homomorphism $\BZ_p\to W(\BZ_p)$ induces ring homomorphisms
$\BZ_p\to Q(\BZ_p)$, $\BZ_p\to Q^{\perf}(\BZ_p)$, and finally, $\BZ_p\to\SW (\BZ_p)$. Combining the latter with the map $\hat W\mono\SW$ from \eqref{e:Akhil's sequence1}, we get a ring homomorphism 
$\BZ_p\oplus\hat W (\BZ_p)\to\SW (\BZ_p)$.
\begin{prop}    \label{p:SW(Z_p)}
This map $\BZ_p\oplus\hat W (\BZ_p)\to\SW (\BZ_p)$ is an isomorphism.
\end{prop}

\begin{proof}
Apply Proposition~\ref{p:admissible rings} of Appendix~\ref{s:SW (R)} to the ring $\BZ/p^n\BZ$.
\end{proof}

\subsection{The map $\hV:\SW\to\SW$}  \label{ss:2hat V}
\subsubsection{The Witt vector $\bu$}   \label{sss:the Witt vector u}
In \S\ref{sss:bar bu} we defined an invertible element $\bar\bu\in W(\BZ_p)/\hat W(\BZ_p)$.
Once and for all, \emph{we fix an element $\bu\in W(\BZ_p)$ such that $\bu\mapsto\bar\bu$}; then $\bu$ is automatically invertible.
The element $\bu$ is unique up to multiplication by an element of $1+\hat W(\BZ_p)\subset W(\BZ_p)$.

By Proposition~\ref{p:SW(Z_p)},
\begin{equation}  \label{e:doesn't matter}
1+\hat W(\BZ_p)\subset \SW (\BZ_p)^\times ,
\end{equation}
so a particular choice of $\bu$ does not really matter for us: it is straightforward to pass from one choice to another (e.g., see \S\ref{sss:independence of bu}).
For any $p$, one can set $\bu:=V^{-1}(p-[p])$, where $[p]\in W(\BZ_p)$ is the Teichm\"uller element. If $p>2$ then $\bar\bu=1$, so one can set $\bu =1$.
According to \cite[Prop.~2.29]{BMVZ}, 
in the case $p=2$ one can take $\bu$ to be the Teichm\"uller element $[-1]$.

For any choice of $\bu$, one has
\begin{equation}  \label{e:2reduction of boldface u}
\bu\in\Ker (W (\BZ_p)^\times\to W (\BF_p)^\times)
\end{equation}
because $V\bu-p\in\hat W (\BZ_p)$ by the definition of $\bu$.

\subsubsection{The map $\hV:W\to W$}  \label{sss:2hat V}
Similarly to \S\ref{sss:hV}, define an additive map $\hV:W\to W$ by 
\begin{equation} 
\hV:=V\circ\bu ;
\end{equation}
in other words, $\hV (x):=V(\bu x)$. Then
\begin{equation} \label{e:F hat V=bp}
\hV F= V(\bu )=\hV (1), \quad F\hV=\bp, \mbox{ where } \bp:=p\bu .
\end{equation}
By \eqref{e:definition of bar bu} and \eqref{e:property of bar bu}, we have
\begin{equation} \label{e:where bp lives}
V(\bu )\in p+\hat W(\BZ_p), \quad \bp\in p+\hat W(\BZ_p).
\end{equation}

\subsubsection{The map $\hV:Q^{\perf}\to Q^{\perf}$}  \label{sss:3hat V}
The map $\hV:Q\to Q$ from \S\ref{sss:hV} commutes with $F:Q\to Q$. So $\hV$ acts on the projective system \eqref{e:Qperf as a projective system}. Therefore we get an additive map
$\hV:Q^{\perf}\to Q^{\perf}$.

\subsubsection{The map $\hV:\SW\to \SW$}  \label{sss:4hat V}
Recall that $\SW:=W\times_Q Q^{\perf}$. So combining the maps $\hV:W\to W$ and $\hV:Q^{\perf}\to Q^{\perf}$ from \S\ref{sss:2hat V}-\ref{sss:3hat V}, we get an additive map $\hV:\SW\to \SW$.
The equalities \eqref{e:F hat V=bp} still hold; note that $V(\bu ),\bp\in\SW (\BZ_p)$ by \eqref{e:where bp lives} and \S\ref{sss:SW(Z_p)}, so $V(\bu )$ and $\bp$ make sense as additive endomorphisms of $\SW$.

The maps $\hV,F:\SW\to \SW$ satisfy the usual identity \eqref{e:usual identity}. This is also true for the maps $\hV,F:W\to W$ and $\hV,F:Q^{\perf}\to Q^{\perf}$.

\begin{lem}  \label{l:kernel and cokernel}
For every $n\in\BN$ one has exact sequences
\begin{equation}   \label{e:cokernel of hV^n}
0\to \SW\overset{\hV^n}\longrightarrow \SW\to W_n\to 0,
\end{equation}
\begin{equation}   \label{e:kernel of F^n}
0\to \hat W^{(F^n)}\to\SW\overset{F^n}\longrightarrow  \SW\to 0,
\end{equation}
in which the maps $\SW\to W_n$ and $\hat W^{(F^n)}\to\SW$ come from the canonical maps $\SW\epi W$ and $\hat W\mono\SW$ from \eqref{e:Akhil's sequence1}-\eqref{e:Akhil's sequence2}.
\end{lem}

\begin{proof}
Exactness of \eqref{e:cokernel of hV^n} follows from \eqref{e:Akhil's sequence2} and the fact  that $\hV :T_F(Q)\to T_F (Q)$ is an isomorphism by Corollary~\ref{c:2acyclicity of bicomplex}.
Exactness of \eqref{e:kernel of F^n} follows from \eqref{e:Akhil's sequence1} and the fact that $F:Q^{\perf}\to Q^{\perf}$ is an isomorphism (by the definition of $Q^{\perf}$); it also follows directly from~\eqref{e:3check W directly}.
\end{proof}

\subsubsection{Remarks}   \label{sss:warning to myself}
(i) By \eqref{e:F hat V=bp}, for every $n\in\BN$ one has
\[
F^n\hV^n=p^n\bu_n , \quad\mbox{ where } \bu_n:=\prod_{i=0}^{n-1} F^i(\bu ).
\]

(ii) As said in \S\ref{sss:the Witt vector u}, one can set $\bu$ to be $1$ if $p>2$ and $[-1]$ if $p=2$. For this choice of $\bu$ one has $F(\bu )=1$ and $\bu_n=\bu$.
So for \emph{any} choice of $\bu$ one has $\bu_n/\bu\in 1+\hat W(\BZ_p)$.

\subsection{The operators $F,\hV$ on $W\times_{Q}W^{\perf}$}   \label{ss:hV on the fiber product}
In this subsection (which can be skipped by the reader) we explain a way to think about $\hV :\SW\to\SW$ (see \S\ref{sss:avoiding Q}).

\subsubsection{}   \label{sss:hV on the fiber product}
The map $F:W^{\perf}\to W^{\perf}$ is clear; it is invertible. We \emph{define} $\hV:W^{\perf}\to W^{\perf}$ by $\hV:=pF^{-1}$.
So the map $W^{\perf}\to W$ is \emph{not} $\hV$-equivariant.

On the other hand, the map $W^{\perf}\to Q^{\perf}$ is $\hV$-equivariant:  indeed, the map $$\hV :Q^{\perf}\to Q^{\perf}$$
equals $pF^{-1}$ by \eqref{e:F times hV}. So the composite map $W^{\perf}\to Q^{\perf}\to Q$ is $\hV$-equivariant.

\subsubsection{}  \label{sss:hV on another fiber product}
Combining the operators $F,\hV$ from \S\ref{sss:hV on the fiber product} with the operators $F,\hV$ on $W$, we get a ring homomorphism
\[
F:W\times_{Q}W^{\perf}\to W\times_{Q}W^{\perf}
\]
and an additive homomorphism
\[
\hV:W\times_{Q}W^{\perf}\to W\times_{Q}W^{\perf};
\]
these operators satisfy the usual relation  \eqref{e:usual identity}. Moreover, $F\hV=\bp'$, where
\[
\bp':=(\bp ,p)\in (W\times_{Q}W^{\perf} )(\BZ_p).
\]

The map $W\times_{Q}W^{\perf}\to W\times_{Q}Q^{\perf}=\SW$ commutes with $F$ and $\hV$.

\subsubsection{Remark}   \label{sss:rewriting}
Using the isomorphism $W^{\perf}\ltimes\hat W\iso  W\times_{Q}W^{\perf}$ defined at the end of \S\ref{sss:Some exact sequences},
 one can rewrite the maps $F,\hV$ from \S\ref{sss:hV on another fiber product} as operators acting on $W^{\perf}\ltimes\hat W$, but
 the formula for the operator $\hV$ on $W^{\perf}\ltimes\hat W$ is a bit ugly.

\subsubsection{A way to avoid $Q$}   \label{sss:avoiding Q}
Possibly some readers would like to have a definition of $\SW$ and the operators $F,\hV$ on $\SW$ which does not use the space $Q:=W/\hat W$.
Formula~\eqref{e:3check W directly} for $\SW$ does not involve $Q$, but it is inconvenient for defining $\hV:\SW\to\SW$. However, one can define $\SW$ 
by the exact sequence
\[
0\to \hat W^{\perf}\to W\times_Q W^{\perf}\to\SW\to 0
\]
(which already appeared in \S\ref{sss:Some exact sequences}, see \eqref{e:Akhil's sequence3}) and think of $W\times_Q W^{\perf}$ as the ind-scheme
\[
\{ (x,y)\in W\times W^{\perf}\,|\, \pi (y)-x\in\hat W\},
\]
where $\pi:W^{\perf}\to W$ is the canonical map. Then the operator $\hV :\SW\to\SW$ comes from the operator $\hV$ on $W\times_Q W^{\perf}$, and the latter is very explicit: it takes $(x,y)\in W\times_Q W^{\perf}$ to
$(\hV (x),pF^{-1}(y))$.

\subsection{$\SW$ as a $\delta$-ring space} \label{ss:delta-ring structure} 
Recall that $\SW:=W\times_Q Q^{\perf}$. Each of the spaces $W,Q, Q^{\perf}$ is a $\delta$-ring space (this is stronger than being a ring space with a ring endomorphism $F$).
So $\SW$ is a $\delta$-ring space. The operators $\delta$ and $\hV$ acting on $\SW$ (or on $W$, $Q$, $Q^{\perf}$)
satisfy the identity
\[
\delta (\hV (x ))=\hV (1)^{p-1}\cdot \hV (\delta (x))+ x\cdot\delta (\hV (1)).
\]
More details can be found in \cite[\S 5]{BMVZ}, 
where it is assumed that the element $\bu\in W(\BZ_p)$ from \S\ref{sss:the Witt vector u} is chosen in a particular way (namely, $\bu =1$ if $p>2$, $\bu =[-1]$ if $p=2$).

\subsection{Derived $p$-completeness of $\SW$}  \label{ss:derived completeness}
The general notion of derived $p$-completeness is recalled in Appendix~\ref{s:derived completeness}.

\begin{lem}  \label{l:derived completeness of SW}
$\SW$ is derived $p$-complete.
\end{lem}

This is \cite[Prop.~3.38]{BMVZ}. Before giving the proof, let us make the following remark. 

\begin{rem} \label{r:naive derived p-completeness is good}
$\SW$ is a sheaf on $\pNilp^{\op}$ equipped with the fpqc topology. A product of exact sequences of fpqc sheaves is exact, so by Lemma~\ref{l:repleteness} from Appendix~\ref{s:derived completeness}, derived $p$-completeness of a sheaf $\cF$ on $\pNilp^{\op}$
just means that $\cF (R)$ is derived $p$-complete for each $R\in\pNilp$.
\end{rem}
  
\subsubsection{Proof of Lemma~\ref{l:derived completeness of SW}}
We follow \cite{BMVZ}. The exact sequence \eqref{e:Akhil's sequence2} shows that it suffices to prove derived $p$-completeness of $W$ and $T_F(Q)$. By Remark~\ref{r:naive derived p-completeness is good}, this amounts to proving derived $p$-completeness of $W (R)$ and $(T_F(Q))(R)$ for each $R\in\pNilp^{\op}$.
For $W (R)$, this follows from \S\ref{sss:2derived completeness for Z-modules}(iii). By definition, $T_F(Q)$ is the projective limit of the sheaves $Q^{(F^n)}$. By \eqref{e:F times hV}, $p^nQ^{(F^n)}=0$. So $(T_F(Q))(R)$ is derived $p$-complete by \S\ref{sss:2derived completeness for Z-modules}(i-ii).
\qed 

\subsection{The operator $1-\hV :\SW\to\SW$}   \label{ss:A simple computation}
The remaining part of \S\ref{s:check W}  can be skipped by the reader.

\subsubsection{The goal}   \label{sss:not V-complete}
As before, $W$ and $\SW$ are considered as sheaves on $\pNilp^{\op}$. The map $1-\hV :W\to W$ is invertible: its inverse is $\sum\limits_{n=0}^\infty\hV^n$. So
\begin{equation}   \label{e:1-V on W}
\Ker (W\overset{1-\hV}\longrightarrow W)=\Coker (W\overset{1-\hV}\longrightarrow W)=0.
\end{equation}
On the other hand, in \S\ref{sss:Cone(1-hV)} we will show that
\begin{equation}  \label{e:Ker (1-V) on check W}
\Ker (\SW\overset{1-\hV}\longrightarrow\SW )=\BZ_p (1),
\end{equation}
\begin{equation}     \label{e:Coker (1-V) on check W}
\Coker (\SW\overset{1-\hV}\longrightarrow \SW )=0.
\end{equation}

Let us note that $\Ker (\SW\overset{1-F}\longrightarrow\SW )=\BZ_p$, see \S\ref{sss:Ker (1-F)} below.

Formula~\eqref{e:Ker (1-V) on check W} shows that $\SW$ is \emph{not $\hV$-complete} (unlike $W$). This is \emph{good}: the kernel of $1-\hV:\SW\to\SW$ is quite meaningful.

\begin{lem}   \label{l:Cone(1-hV)}
One has a canonical isomorphism 
\begin{equation}  \label{e:Cone (1-hV)}
\Cone(\hat W\overset{1-\hV}\longrightarrow \hat W)\iso\hat\BG_m,
\end{equation}
where $\hat\BG_m$ is the formal multiplicative group.
\end{lem}

\begin{proof}
Recall that $\hV:=V\circ\bu$, where $\bu\in\Ker (W (\BZ_p)^\times\to W (\BF_p)^\times)$, see \eqref{e:2reduction of boldface u}. There is a unique $\beta\in\Ker (W (\BZ_p)^\times\to W (\BF_p)^\times)$ such that $\beta/F(\beta)=\bu$: namely,
\begin{equation}   \label{e:beta}
\beta=\prod\limits_{i=0}^\infty F^i(\bu )
\end{equation}
(the infinite product converges). The commutative diagram
\[
\xymatrix{
\hat W\ar[r]^{1-\hV} \ar[d]_\beta & \hat W\ar[d]^\beta\\
\hat W\ar[r]^{1-V} & \hat W
}
\]
defines an isomorphism between $\Cone(\hat W\overset{1-\hV}\longrightarrow \hat W)$ and the complex $\Cone(\hat W\overset{1-V}\longrightarrow \hat W)$. It is clear that $\Ker (\hat W\overset{1-V}\longrightarrow \hat W)=0$. It is well known\footnote{E.g., see \cite[formula (211)]{Zink}.} that
the map
\begin{equation}  \label{e:lambda}
\lambda: \hat W\to\hat\BG_m, \quad \lambda(\sum_{i=0}^\infty V^i[x_i]):=\prod_{i=0}^\infty \exp (x_i+\frac{x_i^p}{p}+\frac{x_i^{p^2}}{p^2}+\ldots )
\end{equation}
induces an isomorphism $\Coker (\hat W\overset{1-V}\longrightarrow \hat W)\iso \hat\BG_m$. 
\end{proof}

Let us note that the isomorphism \eqref{e:Cone (1-hV)} constructed in the proof of Lemma~\ref{l:Cone(1-hV)} comes from the map
\begin{equation}  \label{e:tilde lambda}
\tilde\lambda :\hat W\to\hat\BG_m, \quad  \tilde\lambda (x):=\lambda (\beta x),
\end{equation}
where $\lambda$ is given by  \eqref{e:lambda} and $\beta\in W(\BZ_p)^\times$ is given by \eqref{e:beta}.

\begin{lem}
One has canonical exact sequences
\begin{equation}  \label{e:1}
0\to\hat\BG_m\to Q\overset{1-\hV}\longrightarrow Q\to 0,
\end{equation}
\begin{equation}  \label{e:2}
0\to\mu_{p^n}\to Q^{(F^n)}\overset{1-\hV}\longrightarrow Q^{(F^n)}\to 0,
\end{equation}
\begin{equation}  \label{e:3}
0\to\BZ_p(1)\to T_F(Q)\overset{1-\hV}\longrightarrow T_F(Q)\to 0.
\end{equation}
\end{lem}

\begin{proof}
$Q=W/\hat W$, so \eqref{e:1} follows from Lemma~\ref{l:Cone(1-hV)} and formula~\eqref{e:1-V on W}.
The morphism $F:Q\to Q$ is surjective, and its restriction to $\Ker (Q\overset{1-\hV}\longrightarrow Q)$ equals $F\hV=p$. So \eqref{e:2} and \eqref{e:3} follow from \eqref{e:1}.
\end{proof}

\subsubsection{Proof of \eqref{e:Ker (1-V) on check W}-\eqref{e:Coker (1-V) on check W}}   \label{sss:Cone(1-hV)}
$\SW$ is an extension of $W$ by $T_F(Q)$, so \eqref{e:Ker (1-V) on check W} and \eqref{e:Coker (1-V) on check W} follow from \eqref{e:1-V on W} and the exact sequence \eqref{e:3}. \qed

\subsubsection{Remark}   \label{sss:Ker (1-F)}
The map $1-F:T_F(Q)\to T_F(Q)$ is invertible because $F:Q^{(F^n)}\to Q^{(F^n)}$ is nilpotent. So
\begin{equation}    \label{e:Cone (1-F) on check W}
\Cone(\SW\overset{1-F}\longrightarrow \SW )=\Cone(W\overset{1-F}\longrightarrow W)=\BZ_p[1].
\end{equation}

\subsection{An autoduality conjecture}  \label{ss:autoduality conjecture}
\subsubsection{}
Let $\alpha\in\Ext (\SW,\BZ_p(1))$ be the class of the extension
\begin{equation}  \label{e:the extension of SW}
0\to\BZ_p(1)\to\SW\overset{1-\hV}\longrightarrow\SW\to 0
\end{equation}
provided by formulas \eqref{e:Ker (1-V) on check W}-\eqref{e:Coker (1-V) on check W}. Clearly
\begin{equation}  \label{e:alpha fixed by hV}
(1-\hV^*)(\alpha)=0,
\end{equation}
where $\hV^*:\Ext (\SW,\BZ_p(1))\to\Ext (\SW,\BZ_p(1))$ is induced by $\hV :\SW\to\SW$.

\subsubsection{}
Let $R\in\pNilp$. Let $\SW_R:=\SW\times\Spec R$ (i.e., $\SW_R$ is the base change of $\SW$ to $R$).
Let $\alpha_R\in\Ext (\SW_R,\BZ_p(1)_R)$ be the image of $\alpha\in\Ext (\SW,\BZ_p(1))$. Let
\begin{equation}  \label{e:the map f_R}
f_R: \SW (R)\to\Ext (\SW_R,\BZ_p(1)_R)
\end{equation}
be the unique $W(R)$-linear map such that $f_R(1)=\alpha_R$.

\begin{conj}   \label{conj:autoduality}
The map \eqref{e:the map f_R} is an isomorphsim.
\end{conj}

\subsubsection{Motivation of the conjecture}
In \S\ref{s:economic models for hat sR_n} we will show that if $p>2$ then
\begin{equation}  \label{e:self-dual description of Cone}
\Cone (\SW\overset{p^n}\longrightarrow\SW)\simeq\Cone (\hat W^{(F^n)}\longrightarrow  W_n),
\end{equation}
see \S\ref{sss:I_{n,m}/I'_{n,m} in some cases}(i).
Conjecture~\ref{conj:autoduality} is motivated by this formula and by Cartier duality between $\hat W^{(F^n)}$ and $W_n$. 

If $p$ is \emph{any} prime then $\Cone (\SW\overset{p^n}\longrightarrow\SW)$ has a description (see \S\ref{s:tilde A_n}), which is more complicated then \eqref{e:self-dual description of Cone}
but still self-dual in some sense. 

\subsubsection{Remarks}
(i) By  Proposition~\ref{p:HHom=0} below, $\Hom (\SW_R,\BZ_p(1)_R)=0$.

(ii) The following lemma shows that the map \eqref{e:the map f_R} interchanges $F$ and $\hV$.

\begin{lem}
$f_R\circ F=\hV^*\circ f_R$, and $f_R\circ\hV=F^*\circ f_R$.
\end{lem}

The proof is based on \eqref{e:alpha fixed by hV}.

\begin{proof}
(i) If $a\in\SW(R)$ then $f_R(F(a))=F(a)^*\alpha_R=F(a)^*\hV^*\alpha_R=\hV^*(a^*\alpha_R)=\hV^*(f_R(a))$.

(ii) $f_R(\hV (a))=\hV(a)^*\alpha_R=(\hV a F)^*\alpha_R=F^*a^*\hV^*\alpha_R=F^*(a^*\alpha_R)=F^*(f_R(a))$.
\end{proof}

\begin{prop}    \label{p:HHom=0}
$\HHom (\SW ,\BZ_p(1) )=0$, where $\HHom$ stands for the sheaf of $\Hom$'s on~$\pNilp^{\op}$.
\end{prop}

\begin{proof}
We will use well known facts about Cartier duality between $W$ and $\hat W$ (e.g., see \cite[Appendix A]{On the Lau} and references therein).

The exact sequence \eqref{e:Akhil's sequence2} shows that it suffices to prove that
\begin{equation}  \label{e:again HHom=0}
\HHom (W,\BZ_p (1))=0.
\end{equation}
\begin{equation} \label{e:HHom=0}
\HHom (T_F(Q),\BZ_p (1))=0,
\end{equation}

By Cartier duality between $W$ and $\hat W$, we have $\HHom (W,\BZ_p (1))=\HHom (\BQ_p/\BZ_p, \hat W).$
But $\HHom (\BQ_p/\BZ_p, \hat W)=0$ because if $p^n=0$ in $R$ then $p^n\hat W_R\subset V(\hat W_R)$. This proves \eqref{e:again HHom=0}.

To prove \eqref{e:HHom=0}, it suffices to show that $\HHom (T_F(Q),\BG_m)=0$.
Since $F:\hat W\to\hat W$ is surjective, $Q^{(F^n)}=W^{(F^n)}/\hat W^{(F^n)}$ and $T_F(Q)=T_F(W)/T_F(\hat W)$. We have
\[
\HHom (T_F(W),\BG_m)=\limton\HHom (W^{(F^n)},\BG_m)
\]
because the functors $W^{(F^n)}$ are affine schemes. So
\[
\HHom (T_F(Q),\BG_m)=\limton\HHom (Q^{(F^n)},\BG_m).
\]
It remains to show that $\HHom (Q^{(F^n)},\BG_m)=0$. Indeed, we have $\HHom (\hat W^{(F^n)},\BG_m)=W_n$, $\HHom (W^{(F^n)},\BG_m)=\hat W_n$, and
$\Ker (\hat W_n\to W_n)=0$.
\end{proof}

\section{The $\BZ$-graded ring space $\SW^\oplus$}   \label{s:check W oplus}
In \S\ref{s:check W} we defined a ring space $\SW$ and maps $F,\hV:\SW\to\SW$.
The graded ring space $\SW^\oplus$ will be obtained from the triple $(\SW, F,\hV)$ by applying a general algebraic construction descirbed in the next subsection.

\subsection{The Lau equivalence} \label{ss:Lau equivalence} 
In this subsection we retell a part of E.~Lau's paper \cite{Lau21} (but not quite literally).

\subsubsection{The category $\cC$}  \label{sss:triples A,t,u} \label{sss:cC non-economic}
Let $\cC$ be the category of triples $(A,t,u)$, where $A=\bigoplus\limits _{i\in\BZ}A_i$ is a $\BZ$-graded ring and $t\in A_{-1}$, $u\in A_1$ are such that

(i) multiplication by $u$ induces an isomorphism $A_i\iso A_{i+1}$ for $i\ge 1$;

(ii) multiplication by $t$ induces an isomorphism $A_i\iso A_{i-1}$ for $i\le 0$.

Because of (i) and (ii), $\cC$ has an ``economic'' description. To formulate it, we will define a category $\cC^{\ec}$ (where ``ec'' stands for ``economic'') and construct an equivalence $\cC\iso\cC^{\ec}$.

\subsubsection{The category $\cC^{\ec}$} \label{sss:cC economic}
Let $\cC^{\ec}$ be the category of diagrams 
\begin{equation}   \label{e:objects of C ec}
A_0\underset{V}{\overset{F}\rightleftarrows} A_1,
\end{equation}
where $A_0$ and $A_1$ are rings, $F$ is a ring homomorphism, and $V$ is an additive map such that
\begin{equation}   \label{e:aVb}
a\cdot V(a')=V(F(a)a') \quad  \mbox{ for } a\in A_0, \, a'\in A_1
\end{equation}
and for $a'\in A_1$ we have
\begin{equation}   \label{e:FV=p}
F(V(a'))=\bbp a', \quad  \mbox{ where } {\bbp}:=F(V(1))\in A_1.
\end{equation}
Note that by \eqref{e:aVb} we have $VF=V(1)$, which implies \eqref{e:FV=p} if $a'\in F(A_0)$ (but not in general).

\subsubsection{The functor $\cC\to\cC^{\ec}$}   \label{sss:from cC to economic category}
Given a triple $(A,t,u)\in \cC$, we construct a diagram \eqref{e:objects of C ec} as follows:

(i)  $A_0$ is the $0$-th graded component of $A$;

(ii) $A_1$ is the first graded component of $A$, and the product of $x,y\in A_1$ is as follows: first multiply $x$ by $y$ in $A$, then apply the isomorphism $A_2\iso A_1$ inverse to $u:A_1\iso A_2$; equivalently,
the product in $A_1$ comes from the product in $A/(u-1)A$ and the natural map $A_1\to A/(u-1)A$, which is an isomorphism by virtue of \S\ref{sss:triples A,t,u}(i);

(iii) $F:A_0\to A_1$ is multiplication by $u$, and $V:A_1\to A_0$ is multiplication by $t$.

\begin{prop}    \label{p:Lau equivalence}
The above functor $\cC\to\cC^{\ec}$ is an equivalence. The inverse functor $\fL :\cC^{\ec}\to\cC$ takes a diagram $A_0\underset{V}{\overset{F}\rightleftarrows} A_1$ to a certain graded subring of the graded ring 
$$A_0[t,t^{-1}]\times A_1[u,u^{-1}], \quad \deg t:=-1, \; \deg u=1;$$
namely, the $i$-th graded component of the subring is the set of pairs $(at^{-i},a'u^i)$, where $a\in A_0$ and $a'\in A_1$ satisfy the relation
\begin{equation}  \label{e:Lau equivalence}
a'=\bbp^{-i}F(a) \mbox{ if } i\le 0, \quad a=V(\bbp^{i-1}a')  \mbox{ if } i>0.
\end{equation}
(As before, $\bbp:=F(V(1))\in A_1$.) \qed
\end{prop}

The functor $\fL :\cC^{\ec}\to\cC$ will be called the \emph{Lau equivalence.}

The proof of the proposition is left to the reader. However, let us make some remarks.

\subsubsection{Remarks}   \label{sss:idea behind Lau equivalence}
(i) The description of $\fL$ from Proposition~\ref{p:Lau equivalence} is motivated by the following observation:
if $(A,t,u)\in \cC$ then the natural map $A\to A[1/t]\times A[1/u]$ is injective, $A[1/t]=A_0[t,t^{-1}]$, and $A[1/u]=A_1[u,u^{-1}]$, where the ring structure on $A_1$ is as in \S\ref{sss:from cC to economic category}(ii).

(ii) If $(A,t,u)\in \cC$ then the nonpositively graded part of $A$ identifies with $A_0[t]$ and the positively graded one identifies with $uA_1[u]$, where the ring structure on $A_1$ is as in \S\ref{sss:from cC to economic category}(ii).
So instead of describing $A$ as a subring of $A_0[t,t^{-1}]\times A_1[u,u^{-1}]$, one could describe $A$ as the group $A[t]\oplus uA_1[u]$ equipped with a ``tricky'' multiplication operation.

\subsubsection{Remarks (to be used in \S\ref{sss:independence of bu})}   \label{sss:changing V}
(i) The functor $\cC\to\cC^{\ec}$ from \S\ref{sss:from cC to economic category} can also be described as follows: it takes $(A,t,u)\in\cC$ to $A_0\underset{V}{\overset{F}\rightleftarrows} R$, where $R$ is the $0$-th graded component of the localization $A[u^{-1}]$,
the map $F:A_0\to R$ comes from the map $A\to A[u^{-1}]$, and $V:R\to A_0$ is the composition of $u:R\iso A_1$ and $t:A_1\to A_0$.

(ii) Let $A_0\underset{V}{\overset{F}\rightleftarrows} R$ be an object of $\cC^{\ec}$, and let $(A,t,u)\in\cC$ be its image under $\fL$. Let $\alpha\in R^\times$ and $V'=V\circ\alpha$, i.e., $V'(a')=V(\alpha a')$ for all $a'\in R$.
Then $A_0\underset{V'}{\overset{F}\rightleftarrows} R$ is an object of $\cC^{\ec}$, and its image under $\fL$ is canonically isomorphic to $(A,t,\alpha u)$. The latter follows from the description of the functor $\cC\to\cC^{\ec}$ given in (i).

\subsubsection{Some examples from \cite{Lau21}}   \label{sss:Witt frame}
(i) For any ring $R$ the maps $F,V:W(R)\to W(R)$ satisfy the properties from \S\ref{sss:cC economic} (with $\bbp=p$). Applying the Lau equivalence to the diagram $W(R)\underset{V}{\overset{F}\rightleftarrows} W(R)$, one gets an object of $\cC$.
Following \cite{Lau21}, we call it the \emph{Witt frame}. Following \cite{Daniels}, we denote it by $W (R)^\oplus$ (in \cite[ Example 2.1.3]{Lau21} it is denoted by $\underline{W}(R)$).

(ii) Let $n\in\BN$ and let $R$ be an $\BF_p$-algebra. Then we have a map $F:W_n(R)\to W_n(R)$ (in addition to $V:W_n(R)\to W_n(R)$).
Applying the Lau equivalence to the diagram $W_n(R)\underset{V}{\overset{F}\rightleftarrows} W_n(R)$, one gets an object of $\cC$.
Following \cite{Lau21}, we call it the \emph{$n$-truncated Witt frame}. Following \cite{Daniels}, we denote it by $W_n(R)^\oplus$ (in Example~2.1.6 of \cite{Lau21} it is denoted by $\underline{W_n}(R)$).

\medskip

Let us note that $W (R)^\oplus$ and $W_n (R)^\oplus$ are particular examples of ``higher frames'' in the sense of \cite[\S 2]{Lau21}. 

\subsection{Definition of $\SW^\oplus$} \label{ss:check W oplus} 
\subsubsection{Definition}   \label{sss:check W oplus}
For any $R\in\pNilp$, we defined in \S\ref{s:check W} a diagram $\SW (R)\underset{\hV}{\overset{F}\rightleftarrows} \SW (R)$, which is an object of $\cC^{\ec}$.
The image of this diagram under the functor $\fL :\cC^{\ec}\to\cC$ from Proposition~\ref{p:Lau equivalence} is denoted by $\SW^\oplus (R)$. Thus $\SW^\oplus (R)$ is a $\BZ$-graded algebra over the $\BZ$-graded ring $\BZ_p [t,u]$, where $\deg t=-1$, $\deg u=1$. 
The description of $\fL$ given in Proposition~\ref{p:Lau equivalence} yields a canonical monomorphism of $\BZ$-graded rings
\begin{equation}  \label{e:check W oplus maps to the product}
\SW^\oplus\mono \SW [u,u^{-1}]\times \SW [t,t^{-1}] 
\end{equation}
such that the map $\SW^\oplus\to \SW [u,u^{-1}]$ induced by \eqref{e:check W oplus maps to the product} is an isomorphism in positive degrees and
the map $\SW^\oplus\to\SW [t,t^{-1}]$  is an isomorphism in non-positive degrees. In particular, each graded component of $\SW^\oplus$ is isomorphic to $\SW$ as a sheaf of abelian groups.

\subsubsection{Independence of the choice of $\bu$}  \label{sss:independence of bu}
Recall that the operator $\hV$ from \S\ref{ss:2hat V} depends on the choice of $\bu\in W(\BZ_p)$. By \eqref{e:doesn't matter} and \S\ref{sss:changing V}(ii), $\SW^\oplus (R)$ does not depend on this choice up to canonical isomorphism of $\BZ$-graded $\BZ_p [t]$-algebras (rather than of $\BZ_p [t,u]$-algebras).

\section{The ring stacks $\SR_n$ and $\SR_n^\oplus$} \label{s:2the ring stacks}
\subsection{Ring groupoid generalities}  \label{ss:Ring groupoids}
The goal of this subsection is to give references to the basic definitions from the elementary survey \cite{ring groupoid} and to introduce the somewhat nonstandard notation related to cones (see \S\ref{sss:cone notation} below).

\subsubsection{Ring groupoids}    \label{sss:Ring groupoids}
A definition of the $(2,1)$-category $\RGrpds$ of ring groupoids can be found in \cite[\S 2.2.2]{ring groupoid}. Section 3 of \cite{ring groupoid} contains several equivalent definitions (or incarnations) of the 1-category of ring groupoids and the definition of the functor from it to the $(2,1)$-category $\RGrpds$. 
Here are some of the incarnations\footnote{In each case we describe the objects. Morphisms are defined in the most naive way.} of the 1-category of ring groupoids discussed in \cite{ring groupoid}:

(i) groupoids internal to the category of rings (see \cite[\S 3.2.2]{ring groupoid});

(ii) DG rings $A^\bcdot$ with $A^i=0$ for $i\ne 0,-1$ (see \cite[\S 3.3.3]{ring groupoid});

(iii) quasi-ideal pairs (see \cite[\S 3.3.1]{ring groupoid}), i.e., diagrams $I\overset{d}\longrightarrow A$, where $A$ is a ring, $I$ is an $A$-module, and $d:I\to A$ is an $A$-linear map such that $d(x)\cdot y=d(y)\cdot x$ for all $x,y\in I$ (in this situation one says that $(I,d)$ is a quasi-ideal in $A$).

The functor from (iii) to (i) is described in \cite[\S 3.4.7]{ring groupoid}.

\subsubsection{Notation: $\Cone$ and $\cone$}   \label{sss:cone notation}
The object of the $(2,1)$-category of ring groupoids corresponding to a quasi-ideal $I\overset{d}\longrightarrow A$ is denoted by $\Cone (I\overset{d}\longrightarrow A)$. On the other hand, $\cone (I\overset{d}\longrightarrow A)$ will denote a certain object of the \emph{1-category} of DG rings from \S\ref{sss:Ring groupoids}(ii);
namely, the DG ring is the ring $A\oplus I$, where $A$ is in degree 0, $I$ is in degree $-1$, the diffferential is $d:I\to A$, and the multiplication operation is the obvious one.

\subsubsection{Ring stacks}  \label{sss:Ring stacks generalities}
Let $S$ be a site (usually, $S=\pNilp^{\op}$).
A ring stack on $S$ is a prestack on $S$ with values in $\RGrpds$ which happens to be a stack; equivalently, a ring stack on $S$  is a ring object in the $(2,1)$-category of stacks of groupoids on~$S$ (the meaning of these words is explained in \cite[\S 2.4.1]{ring groupoid}).

If $A$ is a sheaf of rings on $S$ and $(I,d)$ is a quasi-ideal in $A$ then the meaning of the notation $\Cone (I\overset{d}\longrightarrow A)$ and $\cone (I\overset{d}\longrightarrow A)$ is similar to \S\ref{sss:cone notation}, in which $S$ was a point.

\subsubsection{$\BZ$-graded ring stacks}  
The story discussed in \S\ref{sss:Ring groupoids}-\ref{sss:Ring stacks generalities} has an analog in which the word ``ring'' is replaced by the words ``$\BZ$-graded ring''. In this case the role of DG rings from \S\ref{sss:Ring groupoids}(ii) is played by \emph{$\BZ$-graded DG rings} (by a $\BZ$-graded DG ring we mean a DG ring equipped with an \emph{additional} $\BZ$-grading).
The role of the category $\Pol$ from \cite[\S 2.1.2]{ring groupoid} and  \cite[\S\S 2.2.2]{ring groupoid} is played by the category of $\BZ$-graded polynomial algebras in which each variable is homogeneous of some degree.

\subsection{Defining $\SR_n$ and $\SR_n^\oplus$}  \label{ss:Defining hat sR_n}
Let $n\in\BN$. We set
\begin{equation}  \label{e:hat sR_n}
\SR_n:=\Cone (\SW\overset{p^n}\longrightarrow\SW ),
\end{equation}
\begin{equation}  \label{e:hat sR_n oplus}
\SR_n^\oplus:=\Cone (\SW^\oplus\overset{p^n}\longrightarrow\SW^\oplus),
\end{equation}
where $\SW$ is as in \S\ref{ss:Definition of check W} and $\SW^\oplus$ is as in \S\ref{sss:check W oplus}. Thus $\SR_n$ is a $\BZ/p^n\BZ$-algebra stack, and
$\SR_n^\oplus$ is a stack of $\BZ$-graded $\BZ/p^n\BZ$-algebras and even a stack of $\BZ$-graded algebras over the $\BZ$-graded ring $(\BZ/p^n\BZ )[t,u]$, where $\deg t=-1$ and $\deg u=1$.
The map \eqref{e:check W oplus maps to the product} induces a canonical homomorphism of $\BZ$-graded $\BZ/p^n\BZ$-algebras
\begin{equation}  \label{e:hat sR_n oplus maps to the product}
\SR_n^\oplus\to \SR_n [u,u^{-1}]\times \SR_n [t,t^{-1}].
\end{equation}

In the following proposition the word ``degree'' refers to the graded ring structure (we are \emph{not} talking about cohomological degrees).

\begin{prop}  \label{p:positive/nonpositive degrees}
The map $\SR_n^\oplus\to \SR_n [u,u^{-1}]$ induced by \eqref{e:hat sR_n oplus maps to the product} is an isomorphism in positive degrees.
The map $\SR_n^\oplus\to \SR_n [t,t^{-1}]$ induced by \eqref{e:hat sR_n oplus maps to the product} is an isomorphism in non-positive degrees.
\end{prop}

\begin{proof}  
Follows from similar properties of the homomorphism \eqref{e:check W oplus maps to the product}.
\end{proof}

\begin{prop}   \label{p:commutes with filtered colimits}
(i) The functor
\[
R\mapsto \SR_n (R), \quad R\in\pNilp
\]
commutes with filtered colimits.

(ii) The same is true for the functor $R\mapsto \SR_n^\oplus (R)$, \, $R\in\pNilp$.
\end{prop}

\begin{proof}
For statement (i) see \cite[Cor.~3.40]{BMVZ}. 
On the other hand, in \S\ref{s:economic models for hat sR_n} we will give a description of $\SR_n$ which makes statement (i) obvious: namely, we will represent $\SR_n$ as $\Cone (H\to W_n)$, where $H$ is a group ind-scheme which is ind-finite over $\Spf\BZ_p$,
see formula \eqref{e:The economic model} and Lemma~\ref{l:the quotient is an ind-scheme}.  

Statement (ii) follows from (i) and Proposition~\ref{p:positive/nonpositive degrees}.
\end{proof}

The ring stacks $\SR_n$ form a projective system. The same is true for~$\SR_n^\oplus$.

\begin{prop}   \label{p:limit of SR_n}
The projective limit of the ring stacks $\SR_n$ equals $\SW$.
\end{prop}

\begin{proof}
Combine Lemma~\ref{l:derived completeness of SW} with Lemma~\ref{l:Picard stacks} from Appendix~\ref{s:derived completeness}.
\end{proof}

\subsubsection{Remark}   \label{sss:limit of SR_n^oplus}
The projective limit of the stacks $\SR_n^\oplus$ is a \emph{sheaf}, which \emph{strictly} con\-tains~$\SW^\oplus$.
This easily follows from derived $p$-completeness of $\SW$ and the fact
that as a sheaf of abelian groups, $\SW^\oplus$ is isomorphic to the direct sum of countably many copies of $\SW$ (see the end of~\S\ref{sss:check W oplus}).

\subsection{A more economic model for $\SR_n$}  \label{ss:A more economic model}
By a \emph{model} for a ring stack we mean its realization as a $\Cone$ of a quasi-ideal. 
In this subsection  we discuss a model for $\SR_n$, which is more economic than the one provided by \eqref{e:2hat sR_n}. We follow \cite{Sheared,Akhil2}, where the case $n=1$ is considered.

\subsubsection{The ring space $W\times_{Q,F^n}Q$}   
Let $W\times_{Q,F^n}Q$ denote the fiber product in the Cartesian square
\begin{equation}  \label{e:the Cartesian square}
\xymatrix{
W\times_{Q,F^n}Q\ar[r] \ar[d] & Q\ar[d]^{F^n}\\
W\ar[r] & Q
}
\end{equation} 
(the lower horizontal arrow is the projection $W\epi W/\hat W=Q$). We have a canonical homomorphism $W\times_{Q,F^{n+1}}Q\overset{(\id,F)}\longrightarrow W\times_{Q,F^n}Q$, and
\begin{equation} \label{e:check W as a limit}
\leftlimit{n} (W\times_{Q,F^n}Q)=\SW .
\end{equation} 

\subsubsection{Remark} \label{sss:remark on surjectivity}
Surjectivity of $F:\hat W\to\hat W$ implies that the map 
\begin{equation}   \label{e:W to W times_Q Q}
W\to W\times_{Q,F^n}Q, \;\quad w\mapsto (F^n w,\bar w)
\end{equation}
is surjective. It induces an isomorphism $W/\hat W^{(F^n)}\iso W\times_{Q,F^n}Q$,
so  \eqref{e:check W as a limit} is a reformulation of \eqref{e:3check W directly}

\subsubsection{A model for $\SR_n$}    \label{sss:A more economic model}
The map $W\times_{Q,F^n}Q\overset{p^n}\longrightarrow W\times_{Q,F^n}Q$ factors as 
$$W\times_{Q,F^n}Q\overset{\pi}\epi W\to W\times_{Q,F^n}Q,$$
where $\pi$ is the projection and the second map is
\begin{equation}  \label{e:A more economic model}
W\overset{(p^n ,\hV^n)}\longrightarrow W\times_{Q,F^n}Q.
\end{equation}
Consider $W$ as a module over $W\times_{Q,F^n}Q$ via $\pi :W\times_{Q,F^n}Q\epi W$, then \eqref{e:A more economic model} is a quasi-ideal. Let
\[
A_n:=\cone (W\overset{(p^n ,\hV^n)}\longrightarrow W\times_{Q,F^n}Q).
\]

Let us show that the DG ring $A_n$ is a model for $\SR_n$. By construction, $A_n$ 
is a quotient of the DG ring $\cone (W\times_{Q,F^n}Q\overset{p^n}\longrightarrow W\times_{Q,F^n}Q)$; the latter is a 
quotient of $\cone (\SW\overset{p^n}\longrightarrow \SW )$ by \eqref{e:check W as a limit}. 
Thus the DG ring $\cone (\SW\overset{p^n}\longrightarrow \SW )$ maps onto $A_n$.

\begin{prop}   \label{p:quasi-isomorphism}
This map is a quasi-isomorphism, so it induces an isomorphism 
\begin{equation}   \label{e:hat sR_n =Cone(model)}
\SR_n:=\Cone (\SW\overset{p^n}\longrightarrow \SW )\iso \Cone (W\overset{(p^n ,\hV^n)}\longrightarrow W\times_{Q,F^n}Q)
\end{equation}
\end{prop}

\begin{proof}
The kernel of the map from $\cone (\SW\overset{p^n}\longrightarrow \SW )$ to $A_n$ equals 
\begin{equation} \label{e:the kernel}
\cone (T_F(Q)\overset{p^n}\longrightarrow B),
\end{equation}
where
$T_F(Q)$ is given by \eqref{e:T_F(Q)} and $B\subset T_F(Q)$ is as follows. A section of $T_F(Q)$ over $R\in\pNilp$ is a collection of elements $q_i\in Q(R)$, $i\in\BZ$, such that $q_0=0$ and $F(q_i)=q_{i-1}$ for all $i$; in these terms,
$B\subset T_F(Q)$ is defined by the condition $q_n=0$. 

The complex \eqref{e:the kernel} is isomorphic to $\cone (T_F(Q)\overset{\hV^n}\longrightarrow T_F(Q))$; the latter complex is acyclic
because the map $\hV :T_F(Q)\to T_F(Q)$ is an isomorphism by Corollary~\ref{c:2acyclicity of bicomplex}.
\end{proof}

\subsubsection{Remarks}   \label{sss:projective system A_n}
(i) Recall that $A_n:=\cone (W\overset{(p^n ,\hV^n)}\longrightarrow W\times_{Q,F^n}Q)$. As $n$ varies, the DG~rings $A_n$ and $\cone (\SW\overset{p^n}\longrightarrow \SW )$
form projective systems: the transition maps are given by the commutative diagrams
\[
\xymatrixcolsep{4pc}\xymatrix{
W\ar[r]^-{(p^{n+1},\hV^{n+1})} \ar[d]_p & W\times_{Q,F^{n+1}}Q\ar[d]^{(\id,F)}& \SW\ar[r]^{p^{n+1}} \ar[d]_p & \SW\ar[d]^\id\\\
W\ar[r]^-{(p^n,\hV^n)}  & W\times_{Q,F^n}Q& \SW\ar[r]^{p^{n}}& \SW
}
\]
The map from $\cone (\SW\overset{p^n}\longrightarrow \SW )$ to $A_n$ 
defined at the end of \S\ref{sss:A more economic model} is a homomorphism of projective systems of DG rings.

(ii) The naive projective limit of the DG rings $A_n:=\cone (W\overset{(p^n ,\hV^n)}\longrightarrow W\times_{Q,F^n}Q)$ equals $\SW$ by~\eqref{e:check W as a limit}. By Lemma~\ref{l:derived completeness of SW} (or by a direct argument), the same is true for the derived projective limit.

\subsection{A more economic model for $\SR_n^\oplus$}  \label{ss:economic model for hatsR_n oplus}
The DG ring $A_n$ from \S\ref{sss:A more economic model} is equipped with an endomorphism $F$ and an operator $\hV$ satisfying the identities from \S\ref{ss:2hat V} (these maps come from the maps $F,\hV :W\to W$ and $F,\hV :Q\to Q$).
Applying to $(A_n,F,\hV )$ a DG version of the Lau equivalence $\fL$ from Proposition~\ref{p:Lau equivalence}, one gets a DG ring $A_n^\oplus$ equipped with an additional $\BZ$-grading. The corresponding $\BZ$-graded ring stack identifies with $\SR_n^\oplus$
by Proposition~\ref{p:quasi-isomorphism}.

\section{A self-dual model for $\SR_n$} \label{s:tilde A_n}
\subsection{Subject of this section}  \label{ss:tilde A_n subject}
In \S\ref{sss:A more economic model} we constructed a model for $\SR_n$, which was denoted by $A_n$.
In this section (which can be skipped by the reader) we define a DG ring ind-scheme $\tilde A_n$ equipped with a surjective quasi-isomorphism $\tilde A_n\to A_n$. 
Thus $\tilde A_n$ is another model for~$\SR_n$. It turns out that \emph{$\tilde A_n$ is self-dual up to cohomological shift}, see \S\ref{ss:Autoduality of tilde A_n} below.

$\tilde A_n$ is equipped with operators $F$ and $\hV$ (so one can use $\tilde A_n$ to construct a model $\tilde A_n^\oplus$ for $\SR_n^\oplus$ similarly to \S\ref{ss:economic model for hatsR_n oplus}).
In \S\ref{ss:action of 1-hV on tilde A_n} we describe $\Coker (1-\hV :\tilde A_n\to\tilde A_n)$; this is related to the description of $\Ker (1-\hV :\SW\to\SW)$ given in \S\ref{ss:A simple computation}.

\subsection{The DG ring $\tilde A_n$}  \label{ss:tilde A_n}
\subsubsection{Definition}    \label{sss:tilde A_n}
We define $\tilde A_n$  by the following diagram whose squares are Cartesian:
\begin{equation}  \label{e:tilde A_n}
\xymatrix{
\tilde A_n^{-1}\ar[r]^d \ar[d] & \tilde A_n^0\ar[d]\ar[r]&W\ar[d]\\
W\ar[r]^-{(p^n ,\hV^n)} & W\times_{Q,F^n}Q\ar[r]^-{\prr}&Q
}
\end{equation} 
(here the map $W\to W/\hat W=Q$ is the canonical homomorphism).
Diagram~\eqref{e:tilde A_n} gives a map from $\tilde A_n$ to $A_n:=\cone (W\overset{(p^n ,\hV^n)}\longrightarrow W\times_{Q,F^n}Q)$; it is a surjective quasi-isomorphism.

\subsubsection{Explicit description}   \label{sss:tilde A_n explicitly}
We have the ring scheme $W^2=W\times W$ over $\Spf\BZ_p$ and the quasi-ideal
\begin{equation}  \label{e:2the dull sDG ring}
 W^2\overset{d}\longrightarrow W^2, \quad \mbox{ where } d(y_1,y_2):=(p^ny_1,y_2).
 \end{equation}
$\tilde A_n$ is a DG subring of $\cone (W^2\overset{d}\longrightarrow W^2)$, namely
\begin{equation}  \label{tilde A_n^0}
 \tilde A_n^0=\{(x_1,x_2)\in W^2\,|\, x_1\equiv F^n x_2 \},
 \end{equation}
\begin{equation} \label{tilde A_n^-1}
\tilde A_n^{-1}=\{(y_1,y_2)\in W^2\,|\, y_2\equiv \hV^n y_1 \}.
 \end{equation}
 Let us explain that \eqref{tilde A_n^0} is just a short way of saying that for every $R\in\pNilp$ one has
\[
\tilde A_n^0 (R):=\{(x_1,x_2)\in W^2(R)\,|\, x_1-F^n x_2\in\hat W(R) \}.
\]
By \eqref{tilde A_n^0}-\eqref{tilde A_n^-1}, $\tilde A_n^0$ and $\tilde A_n^{-1}$ are additively isomorphic to $W\oplus\hat W$; in particular,
$\tilde A_n^0$ and $\tilde A_n^{-1}$ are \emph{ind-schemes} (not merely fpqc sheaves).

\subsubsection{The maps $F, \hV:\tilde A_n\to\tilde A_n$}   \label{sss:F, hV on tilde A_n}
The DG ring $\tilde A_n$ is equipped with an endomorphism $F$ and an operator $\hV$ satisfying the identities from \S\ref{ss:2hat V} (these maps come from $F,\hV :W\to W$).
The  maps $F, \hV:\tilde A_n\to\tilde A_n$ agree with $F, \hV:A_n\to A_n$.
Similarly to \S\ref{ss:economic model for hatsR_n oplus}, one gets a model $\tilde A_n^\oplus$ for $\SR_n^\oplus$ by applying the Lau equivalence to the triple $(\tilde A_n,F, \hV)$.

\subsubsection{}   \label{sss:gamma}
The element $\gamma:=(1,p^n)\in W(\BZ_p)^2$ belongs to $\tilde A_n^{-1}(\BZ_p)$ and satisfies the relations
\[
d(\gamma )=(p^n,p^n)=p^n, \quad F(\gamma )=\gamma .
\]
These relations mean that $\gamma$ defines an $F$-equivariant\footnote{We are assuming that $F$ acts on $\cone (\BZ\overset{p^n}\longrightarrow\BZ )$ as the identity.} homomorphism from $\cone (\BZ\overset{p^n}\longrightarrow\BZ )$ to $\tilde A_n$. The corresponding homomorphism from $\cone (\BZ\overset{p^n}\longrightarrow\BZ )$ to $A_n$ is equal to the one coming from $A_n$ being a quotient of $\cone (W\times_{Q,F^n}Q\overset{p^n}\longrightarrow W\times_{Q,F^n}Q)$,
see the end of \S\ref{sss:A more economic model}.

\subsection{The projective system $\{ \tilde A_n\}$}  \label{ss:projective system formed by tilde A_n}
\subsubsection{The transition maps}   \label{sss:projective system formed by tilde A_n}
We have maps
\begin{equation}  \label{e:transition in degree 0}
\tilde A_{n+1}^0\to \tilde A_n^0, \quad (x_1,x_2)\mapsto (x_1, Fx_2),
\end{equation}
\begin{equation} \label{e:transition in degree -1}
\tilde A_{n+1}^{-1}\to \tilde A_n^{-1}, \quad (y_1,y_2)\mapsto (py_1, Fy_2).
\end{equation}
These maps define a homomorphism of DG rings
\begin{equation}    \label{e:tilde A_n+1 to tilde A_n}
\tilde A_{n+1}\to \tilde A_n,
\end{equation}
which commutes with $F$ and agrees with the homomorphism $A_{n+1}\to A_n$ from \S\ref{sss:projective system A_n}(i). \emph{However, \eqref{e:tilde A_n+1 to tilde A_n} does not commute with $\hV$ on the nose.}

\subsubsection{The limit}   \label{sss:tilde A_infty}
Let $\tilde A_\infty^i$ (resp.~$A_\infty^i$) be the projective limit of $\tilde A_n^i$ (resp.~$A_n^i$). By \S\ref{sss:projective system A_n}(ii), $A_\infty^0=\SW$, $A_\infty^{-1}=0$.
Let us describe $\tilde A_\infty^i$.

We have $\tilde A_\infty^0=W\times_Q W^{\perf}$, and the homomorphism $\tilde A_\infty^0\to A_\infty^0=\SW$ is the canonical map $W\times_Q W^{\perf}\to W\times_Q Q^{\perf}=\SW$.
It is easy to show that
$\tilde A_\infty^{-1}$ identifies with $\hat W^{\perf}$ so that $d:\tilde A_\infty^{-1}\to\tilde A_\infty^0$ becomes the map $\hat W^{\perf}\mono W\times_Q W^{\perf}$ given by $x\mapsto (0,x)$; this map already appeared in \eqref{e:Akhil's sequence3}.

This description of $\tilde A_\infty$ shows that $\tilde A_\infty$ is equipped with operators $F$ and $\hV$ (they were defined in \S\ref{sss:hV on another fiber product}). The homomorphism of DG rings $\tilde A_\infty\to \tilde A_n$ commutes with $F$; \emph{however, it does not commute with $\hV$ on the nose.}

\subsubsection{Remarks} 
(i) the maps \eqref{e:transition in degree 0} are surjective;

(ii) the maps  \eqref{e:transition in degree -1} are not surjective, but $R^1 \limfromn\tilde A_n^{-1}=0$.

\subsection{Autoduality of $\tilde A_n$}  \label{ss:Autoduality of tilde A_n}
The group ind-schemes $\tilde A_n^0$ and $\tilde A_n^{-1}$ defined by \eqref{tilde A_n^0}-\eqref{tilde A_n^-1} are Cartier-dual to each other because they are both isomorphic to $W\oplus\hat W$. Here is a more precise statement.

\begin{prop}      \label{p:Autoduality of tilde A_n}
Define a group homomorphism $\xi :\tilde A_n^{-1}\to\hat\BG_m$ by 
\begin{equation}  \label{e:xi}
\xi (y_1,y_2):=\tilde\lambda (y_2-\hV^ny_1),
\end{equation}
where $\tilde\lambda :\hat W\to\hat\BG_m$ is given by \eqref{e:tilde lambda}. Then

(i) the pairing
\begin{equation}  \label{e:pairing between B and J}
\tilde A_n^0\times \tilde A_n^{-1}\to\BG_m, \quad (x,y)\mapsto \langle x,y\rangle:=\xi (xy)
\end{equation}
identifies $\tilde A_n^{-1}$ with the Cartier dual of $\tilde A_n^0$;

(ii) for $x\in\tilde A_n^0$ and $y\in\tilde A_n^{-1}$ one has
$\langle Fx,y\rangle=\langle x,\hV y\rangle$ and $\langle \hV x,y\rangle=\langle x,Fy\rangle$.
\end{prop}

\begin{proof}
We have 
\begin{equation}   \label{e:tilde lambda preserved by hV}
\tilde\lambda\circ\hV=\tilde\lambda
\end{equation}
 (see the proof of Lemma~\ref{l:Cone(1-hV)}). So $\xi\circ\hV=\xi$. This implies (ii).

To prove (i), let us rewrite \eqref{e:pairing between B and J} as a pairing 
\[
(W\oplus\hat W)\times (W\oplus\hat W)\to\BG_m.
\]
Let $x=(x_1,x_2)\in\tilde A_n^0$, $y=(y_1,y_2)\in\tilde A_n^{-1}$. Then $x_1=F^n x_2+\alpha$, $y_2=\hV^n y_1+\beta$, where $\alpha,\beta\in\hat W$.
We have 
\[
\langle x,y\rangle=\tilde\lambda (x_2y_2-\hV^n(x_1y_1)), \quad x_2y_2-\hV^n(x_1y_1)=x_2\beta-\hV^n (\alpha y_1).
\]
So using \eqref{e:tilde lambda preserved by hV}, we get $\langle x,y\rangle=\tilde\lambda (x_2\beta)\cdot \tilde\lambda (\alpha y_1)^{-1}$.
It remains to show that the pairing
\[
W\times\hat W\to\BG_m , \quad (u,v)\mapsto \tilde\lambda (uv)
\]
identifies $\hat W$ with the Cartier dual of $W$. This follows from a similar property of the pairing $(u,v)\mapsto\lambda (uv)$, which is well known
(see \cite[Appendix A]{On the Lau} and references therein).
\end{proof}

\subsection{The operator $1-\hV :\tilde A_n\to\tilde A_n$}   \label{ss:action of 1-hV on tilde A_n}
As before, let $\tilde\lambda :\hat W\to\hat\BG_m$ be given by \eqref{e:tilde lambda}. 

\begin{prop}      \label{p:action of 1-hV on tilde A_n}
There is a commutative diagram with exact rows
\begin{equation}   \label{e:action of 1-hV on tilde A_n}
\xymatrix{
0\ar[r]&\tilde A_n^{-1}\ar[r]^{1-\hV} \ar[d]_d & \tilde A_n^{-1}\ar[r]^{\xi}\ar[d]_d&\hat\BG_m\ar[r]\ar[d]^p &0 \\
0\ar[r]&\tilde A_n^0\ar[r]^{1-\hV} & \tilde A_n^0\ar[r]^{\nu} & \hat\BG_m\ar[r] &0
}
\end{equation}
where $\xi$ is defined by \eqref{e:xi} and
\begin{equation}   \label{e:nu}
\nu (x_1,x_2):=\tilde\lambda (p^nx_2-\hV^nx_1).
\end{equation}
\end{prop}

\subsubsection{Remarks}
(i) Formula \eqref{e:nu} makes sense because $p^nx_2-\hV^nx_1\in\hat W$: this follows from the fact that $F^nx_2-x_1\in\hat W$ by \eqref{tilde A_n^0}.

(ii) One can rewrite \eqref{e:nu} as $\nu (a)=\xi (a\gamma)$, where $\gamma\in\tilde A_n^{-1}(\BZ_p )$ is as in \S\ref{sss:gamma}.

\subsubsection{Proof of Proposition~\ref{p:action of 1-hV on tilde A_n}}
Commutativity of the right square of \eqref{e:action of 1-hV on tilde A_n} is checked straightforwardly using \eqref{e:2the dull sDG ring}.
The lower row of \eqref{e:action of 1-hV on tilde A_n} is a complex because $\tilde\lambda\circ (1-\hV )=0$.
To prove its exactness, use the $\hV$-equivariant exact sequence
\[
0\to\hat W\to\tilde A_n^0\overset{\pi_2}\longrightarrow W\to 0,
\]
where $\pi_2(x_1,x_2):=x_2$; one also uses Lemma~\ref{l:Cone(1-hV)} and the fact that the map $1-\hV :W\to W$ is an isomorphism. \qed

\begin{cor}
The map $1-\hV :\SR_n\to\SR_n$ is surjective, and its fiber over 0 identifies with $\mu_{p^n}$.
\end{cor}

\begin{proof}
Follows either from Proposition~\ref{p:action of 1-hV on tilde A_n} or from the exact sequence \eqref{e:the extension of SW}.
\end{proof}

\section{$\SR_n$ as a quotient of $W$}  \label{s:hat sR_n as a quotient of W}
In this section we describe a model for $\SR_n$, which represents $\SR_n$ as a quotient of $W$. This model goes back to \cite{Sheared,Akhil2}.

\subsection{The model}  \label{ss:hat sR_n as a quotient of W}
\subsubsection{}  \label{sss:definition of I_n}
Define a quasi-ideal $I_n\overset{d}\longrightarrow W$ by the Cartesian square
\begin{equation} \label{e:definition of I_n}
\xymatrix{
I_n\ar[r]^d \ar[d] & W\ar[d]^{x\mapsto (F^nx,\bar x)}\\
W\ar[r]^-{(p^n ,\hV^n)} & W\times_{Q,F^n}Q   
}
\end{equation} 
whose lower row is \eqref{e:A more economic model}. Let 
\[
B_n:=\cone (I_n\overset{d}\longrightarrow W).
\]
Recall that $A_n:=\cone (W\overset{(p^n ,\hV^n)}\longrightarrow W\times_{Q,F^n}Q)$.

\begin{prop}   \label{p:2quasi-isomorphism}
The map $B_n\to A_n$ corresponding to diagram \eqref{e:definition of I_n} is a quasi-isomorphism.
\end{prop}

\begin{proof}
Follows from surjectivity of the right vertical arrow of  \eqref{e:definition of I_n}, see \S\ref{sss:remark on surjectivity}.
\end{proof}

Combining Propositions~\ref{p:quasi-isomorphism} and \ref{p:2quasi-isomorphism}, we get the following

\begin{cor}  \label{c:is a model}
$B_n$ is a model for $\SR_n$.  \qed
\end{cor}

\subsubsection{Explicit description of $B_n$}  \label{sss:description of I_n}
By \S\ref{sss:definition of I_n}, we have
\begin{equation}   \label{e:description of I_n}
I_n=\{ (x,y)\in W^2\,|\, F^nx=p^ny, \,\, x-\hV^ny\in\hat W\},
\end{equation}
$d:I_n\to W$ is given by $d(x,y)=x$, and the $W$-module structure on $I_n$ is given by
\[
a\cdot (x,y)=(ax,F^n(a)y), \quad \mbox{ where } a\in W, \; (x,y)\in I_n .
\]

\subsection{More about $B_n$}
\subsubsection{The homomorphism $F:B_n\to B_n$}  \label{sss:F:B_n to B_n}
Define a DG ring homomorphism $F:B_n\to B_n$ as follows:

(i) the map $B_n^0\to B_n^0$ is $F:W\to W$;

(ii) in terms of \eqref{e:description of I_n}, the endomorphism of $I_n=B_n^{-1}$ is given by $(x,y)\mapsto (Fx,Fy)$.

Then the homomorphism $B_n\to A_n$ commutes with $F$.

\subsubsection{A drawback of $B_n$}  \label{sss:drawback of B_n}
The advantage of the model $B_n$ is that one can use it to construct \emph{economic} models for $\SR_n$, see \S\ref{s:economic models for hat sR_n} below. But $B_n$ has the following drawback: the map $\hV:A_n\to A_n$ does not lift to an additive endomorphism of $B_n$.
Moreover, since there is no operator $\hV$ acting on $B_n$, it is hard to use $B_n$ to construct a model for $\SR_n^\oplus$.

\subsubsection{The operators $V:B_{n,\BF_p}\to B_{n,\BF_p}$ and $V:A_{n,\BF_p}\to A_{n,\BF_p}$}  \label{sss:V on B_n in characteristic p}
The drawback of $B_n$ mentioned in \S\ref{sss:drawback of B_n} disappears after base change to $\BF_p$ (because in characteristic $p$ one has $FV=VF$). Here are more details.

\medskip

Let $B_{n,\BF_p}:=B_n\times\Spec\BF_p$. Define an additive map $V:B_{n,\BF_p}\to B_{n,\BF_p}$ as follows:

(i) the map $B_n^0\to B_n^0$ is $V:W\to W$;

(ii) in terms of \eqref{e:description of I_n}, the endomorphism of $I_n=B_n^{-1}$ is given by $(x,y)\mapsto (Vx,Vy)$.

\medskip

Let $A_{n,\BF_p}:=A_n\times\Spec\BF_p$. Let $V:A_{n,\BF_p}\to A_{n,\BF_p}$ be induced by $\hV: A_n\to A_n$.
Unlike $\hV$, the operator $V:A_{n,\BF_p}\to A_{n,\BF_p}$ does not depend on the choice of the Witt vector $\bu\in W(\BZ_p)$ from \S\ref{sss:the Witt vector u}; this follows from
\eqref{e:2reduction of boldface u}. Moreover, $V:A_{n,\BF_p}\to A_{n,\BF_p}$ is described by the most naive formulas similar to those from the definition of $V:B_{n,\BF_p}\to B_{n,\BF_p}$.

The homomorphism $B_{n,\BF_p}\to A_{n,\BF_p}$ commutes with $V$.

\subsubsection{The map $B_{n+1}\to B_n$}
Define a DG ring homomorphism $F:B_{n+1}\to B_n$ as follows:

(i) the map $B_{n+1}^0\to B_n^0$ is $F:W\to W$;

(ii) in terms of \eqref{e:description of I_n}, the map from $I_{n+1}=B_{n+1}^{-1}$ to $I_n=B_n^{-1}$ is given by $(x,y)\mapsto (Fx,py)$.

Then the diagram
\[
\xymatrix{
B_{n+1}\ar[r]\ar[d] & B_n\ar[d]\\
A_{n+1}\ar[r] & A_n
}
\]
commutes. Recall that the vertical arrows of this diagram are surjective quasi-isomorphisms.

Let $B_\infty^i$ be the projective limit of $B_n^i$.  
 One checks that $d:B_\infty^{-1}\to B_\infty^0=W^{\perf}$ is injective and $d (B_\infty^{-1})=\leftlimit{n} \hat W^{(F^n)}$. This projective limit already appeared in formula~\eqref{e:3check W directly}.
 
\subsection{A model for $\SR_{n,\BF_p}^\oplus$}   \label{ss:applying Lau equivalence to B_{n,F_p}}
By \S\ref{sss:F:B_n to B_n} and \S\ref{sss:V on B_n in characteristic p}, the DG ring $B_{n,\BF_p}$ is equipped with operators $F,V$ satisfying the usual identities.
So we can apply to $B_{n,\BF_p}$ the Lau equivalence $\fL$ (see Proposition~\ref{p:Lau equivalence}) and get a $\BZ$-graded\footnote{By a $\BZ$-graded DG ring we mean a DG ring equipped with an \emph{additional} $\BZ$-grading.} 
DG ring ind-scheme, which we denote by $B_{n,\BF_p}^\oplus$.

We have a canonical surjective quasi-isomorphism $B_{n,\BF_p}^\oplus\epi A_{n,\BF_p}^\oplus$, where $$A_{n,\BF_p}^\oplus:=A_n^\oplus\times\Spec\BF_p$$ and $A_n^\oplus$ is the model for $\SR_n^\oplus$ constructed in \S\ref{ss:economic model for hatsR_n oplus}.
So $B_{n,\BF_p}^\oplus$ is a model for $\SR_{n,\BF_p}^\oplus$.

\section{$\SR_n$ as a quotient of $W_n$}  \label{s:economic models for hat sR_n}
In this section we follow \cite{Sheared,Vadik's talk}. 

\subsection{The goal}
One has $V^n(\hat W^{(F^m)})\subset \hat W^{(F^{m+n})}$. We will show that $\hat W^{(F^{m+n})}/V^n(\hat W^{(F^m)})$ is an ind-scheme, see Lemma~\ref{l:the quotient is an ind-scheme} and  \S\ref{ss:I_{n,m}/I'_{n,m}}.
Note that if $m=0$ then $\hat W^{(F^m)}=0$, so $\hat W^{(F^{m+n})}/V^n(\hat W^{(F^m)})=\hat W^{(F^n)}$.

Define a number $\delta_p$ by
\begin{equation}    \label{e:defining delta_p}
\delta_p:=0 \quad\mbox{ for } p>2, \quad \delta_2:=1.
\end{equation}
In \S\ref{sss:economic models for hat sR_n} we will show that for any integer $m\ge\delta_p$, the ring stack $\SR_n$ is canonically isomorphic to
\[
\Cone (\hat W^{(F^{m+n})}/V^n(\hat W^{(F^m)})\to W_n). 
\]
Moreover, we will see there that $\SR_n\times\Spec\BF_p$ canonically identifies with
\[
\Cone (\hat W_{\BF_p}^{(F^{m+n})}/V^n(\hat W_{\BF_p}^{(F^m)})\to W_{n,\BF_p})
\]
for \emph{any} integer $m\ge 0$. 

The reader may prefer to disregard the case of mixed characteristic $2$ and assume that $m=0$.

In addition to the above-mentioned models for $\SR_n$, in \S\ref{ss:applying Lau equivalence to C_{n,F_p}} we will construct an economic model for $\SR_{n,\BF_p}^\oplus$ (this is straightforward because the models for $\SR_{n,\BF_p}$ are equipped with operators $F$ and $V$).

\subsection{A variant of the model from \S\ref{ss:hat sR_n as a quotient of W}} \label{ss:A variant}
Let $d:I_n\to W$ be the quasi-ideal from \S\ref{ss:hat sR_n as a quotient of W}.
For each non-negative $m\in\BZ$ define $I_{n,m}$ and $d:I_{n,m}\to W$ by the pullback square 
\[
\xymatrix{
I_{n,m}\ar[r]^d \ar[d] & W\ar[d]^{F^m}\\
I_n\ar[r]^d & W
}
\]
Since $F:W\to W$ is surjective, Corollary~\ref{c:is a model} implies that for each $m$ one has
\begin{equation}  \label{e:A variant of the model}
\SR_n=\Cone (I_{n,m}\overset{d}\longrightarrow W).
\end{equation}
By  \eqref{e:description of I_n}, we have
\begin{equation}  \label{e:I_n,m explicitly}
I_{n,m}=\{ (x,y)\in W^2\,|\, F^{m+n}x=p^ny, \,\, F^mx-\hV^ny\in\hat W\},
\end{equation}
$d:I_{n,m}\to W$ is given by $d(x,y)=x$, and  the $W$-module structure on $I_{n,m}$ is given by
\[
a\cdot (x,y)=(F^{m+n}(a)x,y), \quad \mbox{ where } a\in W, \; (x,y)\in I_{n,m} .
\]

\subsection{Economic models for $\SR_n$}  \label{ss:economic models for hat sR_n}
\subsubsection{}
We keep the notation of \S\ref{ss:A variant}. Let
\begin{equation}  \label{e:defining I'_n,m}
I'_{n,m}:=\{ (x,y)\in W^2\,|\, x\in V^n(W), \; y=F^mV^{-n}x\}.
\end{equation}

\begin{lem}   \label{l:I' is contained in I}
(i) Let $m\ge\delta_p$, where $\delta_p$ is as in \eqref{e:defining delta_p}. Then $I'_{n,m}\subset I_{n,m}$, where $I_{n,m}$ and $I'_{n,m}$ are given by formulas \eqref{e:I_n,m explicitly}-\eqref{e:defining I'_n,m}.

(ii) $I'_{n,m,\BF_p}\subset I_{n,m,\BF_p}$ for {\bf all} $m$. Here $I_{n,m,\BF_p}$ is the base change of $I_{n,m}$ to $\Spec\BF_p$.
\end{lem}

\begin{proof}
To prove (ii), it suffices to note that by \eqref{e:I_n,m explicitly}   and \eqref{e:2reduction of boldface u}, one has the formula
\begin{equation}     \label{e:hat unnecessary over F_p}
I_{n,m,\BF_p}=\{ (x,y)\in W_{\BF_p}^2\,|\, F^{m+n}x=p^ny, \,\, F^mx-V^ny\in\hat W_{\BF_p}\},
\end{equation}
in which $\hV$ does not appear.

To prove (i), one has to check that $(F^mV^n-\hV^nF^m)(W)\subset\hat W$. Equivalently, the problem is to show that $F^mV^n=\hV^nF^m$ if $F,V,\hV$ are considered as operators in $Q=W/\hat W$.
By formula~\eqref{e:F times hV}, $\hV^nF^m=F^m\hV^n$. Recall that $\hV:=V\bar\bu$, so $F^m\hV^n=F^mV^n\cdot\bar\bu F(\bar\bu)\ldots F^{n-1}(\bar\bu)$. To show that this equals $F^mV^n$, use  \eqref{e:when bar bu=1} if $p>2$ and \eqref{e:property of bar bu} if $m\ge 1$.
\end{proof}

The next two lemmas describe $ I_{n,m}/I'_{n,m}$ if $m\ge\delta_p$. Let 
\[
\hat I_{n,m}:=\{ (x,y)\in \hat W^2\,|\, F^{m+n}x=p^ny \},
\] 
\[
\hat I'_{n,m}:= I'_{n,m}\cap\hat I_{n,m}=
\{ (x,y)\in\hat W^2\,|\,   x\in V^n(\hat W), \; y=F^mV^{-n}x\}. 
\]

\begin{lem}   \label{l:I_{n,m}/I'_{n,m}} 
(i) If $m\ge\delta_p$ then the canonical map $\hat I_{n,m}/\hat I'_{n,m}\to  I_{n,m}/I'_{n,m}$ is an isomorphism.

(ii) For any $m\ge 0$ the canonical map $\hat I_{n,m,\BF_p}/\hat I'_{n,m,\BF_p}\to  I_{n,m,\BF_p}/I'_{n,m,\BF_p}$ is an isomorphism.
\end{lem}

\begin{proof}
We will only prove (i); the proof of (ii) is similar. Let 
\begin{equation}  \label{e:Y_{n,m}}
Y_{n,m}:=\{ (x,y)\in W^2\,|\, F^{m+n}x=p^ny, \; x_i=0 \mbox{ for } i\ge n\},
\end{equation}
where $x_i$'s are the components of the Witt vector $x$. Let 
\begin{equation}  \label{e:hat Y_{n,m}}
\hat Y_{n,m}:=Y_{n,m}\cap\hat W^2=\{ (x,y)\in \hat W^2\,|\, F^{m+n}x=p^ny, \; x_i=0 \mbox{ for } i\ge n\}.
\end{equation}
If $m\ge\delta_p$ then the map $Y_{n,m}\cap I_{n,m}\to I_{n,m}/I'_{n,m}$ is an isomorphism. So to prove (i), it suffices to show that the inclusion
$\hat Y_{n,m}\subset Y_{n,m}\cap I_{n,m}$ is an equality.

Indeed, let $(x,y)\in Y_{n,m}(R)\cap I_{n,m}(R)$, where $R$ is a $p$-nilpotent ring. Since $F^{m+n}x=p^ny$, we see that $x_i$ is nilpotent for each $i<n$; on the other hand, $x_i=0$ for $i\ge n$. So
$x\in\hat W(R)$. Since $F^mx-\hV^ny\in\hat W(R)$, we see that $\hV^ny\in\hat W(R)$. So $y\in\hat W(R)$. Thus $(x,y)\in\hat Y_{n,m}(R)$.
\end{proof}

\begin{lem}   \label{l:the quotient is an ind-scheme}
(i) The map 
\begin{equation}  \label{e:the map inducing the iso}
\hat W^{(F^{m+n})}\to \hat I_{n,m}\subset \hat W^2, \quad x\mapsto (x,0)
\end{equation}
induces an isomorphism
\begin{equation}   \label{e:2 iso between two quotients}
\hat W^{(F^{m+n})}/V^n(\hat W^{(F^m)})\iso \hat I_{n,m}/\hat I'_{n,m}.
\end{equation} 

(ii) $\hat W^{(F^{m+n})}/V^n(\hat W^{(F^m)})$ is an ind-finite ind-scheme over $\Spf\BZ_p$ whose group of $\bar\BF_p$-points is zero. This ind-scheme is isomorphic to the ind-scheme $\hat Y_{n,m}$ from formula~\eqref{e:hat Y_{n,m}}.
\end{lem}

Additional information about the group ind-scheme $\hat W^{(F^{m+n})}/V^n(\hat W^{(F^m)})$ can be found in~\S\ref{ss:I_{n,m}/I'_{n,m}}.

\begin{proof}
Statement (i) follows from surjectivity of $F:\hat W\to\hat W$. Statement (ii) follows from~(i) and the fact that the map $\hat Y_{n,m}\to\hat I_{n,m}/\hat I'_{n,m}$ is an isomorphism,
where $\hat Y_{n,m}$ is defined by~\eqref{e:hat Y_{n,m}}.
\end{proof}

\subsubsection{Remark}  \label{sss:the composition explicitly}
The composition $\hat Y_{n,m}\iso\hat I_{n,m}/\hat I'_{n,m}\iso \hat W^{(F^{m+n})}/V^n(\hat W^{(F^m)})$ takes $(x,y)$ to $x-V^nF^{-m}(y)$, where $F^{-m}(y)$ is a point of $\hat W/\hat W^{(F^m)}$; note that $F^{m+n}(x-V^nF^{-m}(y))=F^{m+n}x-p^ny=0$.

\subsubsection{The economic models}   \label{sss:economic models for hat sR_n}
Let $m\ge\delta_p$, where $\delta_p$ is as in \eqref{e:defining delta_p}. Then $I'_{n,m}\subset I_{n,m}$ by Lemma~\ref{l:I' is contained in I}(i).
The map $d:I_{n,m}\to W$ from \S\ref{ss:A variant} induces an isomorphism 
$$I'_{n,m}\iso V^n (W)=\Ker (W\epi W_n).$$
 So for each $m\ge\delta_p$ we have $\SR_n =\Cone (I_{n,m}/I'_{n,m}\overset{d}\longrightarrow W_n)$. By Lemmas~\ref{l:I_{n,m}/I'_{n,m}}(i) and \ref{l:the quotient is an ind-scheme}, this can be rewritten as
\begin{equation} \label{e:The economic model}
\SR_n=\Cone (\hat W^{(F^{m+n})}/V^n(\hat W^{(F^m)})\to W_n),
\end{equation}
where the map from $\hat W^{(F^{m+n})}/V^n(\hat W^{(F^m)})$ to $W_n$ is the tautological one and the action of $W_n$ on $\hat W^{(F^{m+n})}/V^n(\hat W^{(F^m)})$ comes from the obvious action of $W$ on $\hat W^{(F^{m+n})}$.
Similarly, using Lemma~\ref{l:I' is contained in I}(ii) and Lemma~\ref{l:I_{n,m}/I'_{n,m}}(ii),  we see that for \emph{all} $m$ one has
\begin{equation} \label{e:The economic model over F_p}
\SR_{n,\BF_p} =\Cone (\hat W_{\BF_p}^{(F^{m+n})}/V^n(\hat W_{\BF_p}^{(F^m)})\to W_{n,\BF_p}).
\end{equation}

\subsubsection{Examples}  \label{sss:I_{n,m}/I'_{n,m} in some cases}
(i) Let $p>2$. Then one can set $m=0$ in \eqref{e:The economic model}. So
\begin{equation}  \label{e:economic model for m=0}
\SR_n=\Cone (\hat W^{(F^n)}\overset{d}\longrightarrow  W_n),
\end{equation} 
where $d$ is the tautological map. Note  that in the case $p=2$ the r.h.s. of \eqref{e:economic model for m=0} is \emph{not a $\BZ/2^n\BZ$-algebra} (unlike $\SR_n$); this follows from Lemma~\ref{l:the naive ring stack for any p} below.

(ii) Setting $m=0$ in \eqref{e:The economic model over F_p}, we see that for \emph{all} $p$ (including $p=2$) one has
\begin{equation}  \label{e:economic model over F_p for m=0}
\SR_{n,\BF_p}=\Cone (\hat W_{\BF_p}^{(F^n)}\overset{d}\longrightarrow  W_{n,\BF_p}).
\end{equation} 

(iii) Let $n=1$. As noted in \cite{Vadik's talk}, $\hat W^{(F^{m+1})}/V (\hat W^{(F^m)})$ has a description in the spirit of P.~Berthelot \cite{Ber}: namely, the quasi-ideal $\hat W^{(F^{m+1})}/V (\hat W^{(F^m)})$ in $W_1=\BG_a$ identifies with the nilpotent PD~neighborhood of $Z\subset\BG_a$, where $Z$ is the kernel of $\Fr^m$ acting on $\BG_a\otimes\BF_p$. In the case $m=0$ this is well known. In the case $m\ge 1$ (or more generally, $m\ge\delta_p$) this can be deduced from the isomorphism $\hat W^{(F^{m+1})}/V (\hat W^{(F^m)})\iso\hat Y_{1,m}$ from the proof of Lemma~\ref{l:the quotient is an ind-scheme}. Note that 
$\hat Y_{1,m}=\{(x_0,y)\in \BG_a\times\hat W\,|\, [x_0^{p^m}]-Vy\in\hat W^{(F)}\}$, and the map $\hat Y_{1,m}\to\BG_a$ is given by $x_0$, so we have a pullback square of ind-schemes
\[
\xymatrix{
\hat Y_{1,m}\ar[r] \ar[d] & \BG_a\ar[d]^{z\mapsto z^{p^m}}\\
\hat W^{(F)}\ar[r] & \BG_a
}
\]

\subsection{More on $\hat W^{(F^{m+n})}/V^n(\hat W^{(F^m)})$}    \label{ss:I_{n,m}/I'_{n,m}}
By Lemma~\ref{l:the quotient is an ind-scheme}, $\hat W^{(F^{m+n})}/V^n(\hat W^{(F^m)})$ is an ind-scheme.

\begin{prop}   \label{p:On I_{n,m}/I'_{n,m}}
(i) The  ind-scheme $\hat W^{(F^{m+n})}/V^n(\hat W^{(F^m)})$ can be represented as an inductive limit of a diagram
\begin{equation}
\Spf C_1\mono\Spf C_2\mono\ldots ,
\end{equation}
where each $C_i$ is a finite flat $\BZ_p$-algebra.

(ii) The canonical nondegenerate pairing\footnote{E.g., see \cite[Appendix A]{On the Lau} and references therein.} $\hat W\times W\to\BG_m$ induces isomorphisms
\begin{equation}  \label{e:G_n,m as a Cartier dual}
G_{n,m}\iso\HHom (\hat W^{(F^{m+n})}/V^n(\hat W^{(F^m)}),\BG_m),
\end{equation}
\begin{equation}  \label{e:Cartier dual of G_n,m}
\hat W^{(F^{m+n})}/V^n(\hat W^{(F^m)})\iso \HHom (G_{n,m},\BG_m) ,
\end{equation}
where $G_{n,m}:=\Ker (W_{m+n}\overset{F^n}\epi W_m)$.
\end{prop}

Note that $G_{n,m}$ is a flat affine group scheme of finite type over $\Spf\BZ_p$; it is smooth if and only if $m=0$.

\begin{proof}
We will use the ind-scheme $\hat Y_{n,m}$ defined by~\eqref{e:hat Y_{n,m}}. (If $m=0$ this is not necessary.)

(i) As explained in the proof of Lemma~\ref{l:the quotient is an ind-scheme}, there is an isomorphism of ind-schemes $\hat W^{(F^{m+n})}/V^n(\hat W^{(F^m)})\iso\hat Y_{n,m}$.
We also have the affine scheme $Y_{n,m}$ over $\Spf\BZ_p$ defined by~\eqref{e:Y_{n,m}}. 

The relation between $Y_{n,m}$ and $\hat Y_{n,m}$ is as follows. Let $A$ be the coordinate ring of $Y_{n,m}$. The group $\BG_m$ acts on $Y_{n,m}$: namely, $\lambda\in\BG_m$ acts by
$(x,y)\mapsto ([\lambda] x,[\lambda^{p^{m+n}}] y)$). This action defines a $\BZ$-grading\footnote{$A$ is $p$-complete, and the word ``grading'' is understood in the $p$-complete sense.} on $A$. This grading is non-negative, and each graded component $A_i$ is a finitely generated $\BZ_p$-module.
Then $\hat Y_{n,m}$ is the inductive limit of the closed subschemes $Y_{n,m}^{< r}\subset Y_{n,m}$, where $Y_{n,m}^{< r}$ corresponds to the ideal $\bigoplus\limits_{i\ge r} A_i\subset A$ (here $\bigoplus$ stands for the $p$-completed direct sum).

So it remains to show that each $A_i$ is flat over $\BZ_p$. This is equivalent to flatness of $Y_{n,m}$ over $\Spf\BZ_p$. By~\eqref{e:Y_{n,m}}, $Y_{n,m}$ is the kernel of the composite map
\[
W\times W\overset{f}\longrightarrow W\times W\overset{g}\longrightarrow W\times W\overset{\prr_1}\longrightarrow W, \quad \mbox{where } f:=F^{m+n}\times\id , \;  g(z,y):=(z-p^ny,y).
\]
Clearly $g$ is an isomorphism, and $\prr_1$ is flat. Finally, $f$ is flat because $F:W \to W$ is flat\footnote{One can use \cite[\S 3.4]{Prismatization}. On the other hand, since our $W$ is over $\Spf\BZ_p$ (rather than over $\Spec\BZ_p$), it suffices to note that $F:W_{\BF_p}\to W_{\BF_p}$ is flat.}.

(ii) The isomorphism \eqref{e:G_n,m as a Cartier dual} is straightforward. Statement (i) ensures that the map from $\hat W^{(F^{m+n})}/V^n(\hat W^{(F^m)})$ to its double Cartier dual\footnote{Cartier duality between commutative affine group schemes and commutative ind-affine group ind-schemes over arbitrary bases is discussed in \cite{AM} and also in the lecture \cite{Akhil3} at 46:00. The material from \cite{Akhil3} is closely related to the Appendices of \cite{Bou}.} is an isomorphism.
So \eqref{e:Cartier dual of G_n,m} follows from~\eqref{e:G_n,m as a Cartier dual}.
\end{proof}

\subsection{An economic model for $\SR_{n,\BF_p}^\oplus$}  \label{ss:applying Lau equivalence to C_{n,F_p}}
As before, we use the subscript $\BF_p$ to denote base change to $\Spec\BF_p$.

\subsubsection{The model} 
By \S\ref{sss:I_{n,m}/I'_{n,m} in some cases}(ii), the DG ring ind-scheme
\begin{equation}  \label{e:again economic model over F_p}
C_{n,\BF_p}:=\cone (\hat W^{(F^n)}_{\BF_p}\overset{d}\longrightarrow W_{n,\BF_p})
\end{equation} 
is a model for $\SR_{n,\BF_p}$. The DG ring~\eqref{e:again economic model over F_p} is equipped with operators $F,V$ satisfying the usual identities; they come from the maps $F,V:W_{\BF_p}\to W_{\BF_p}$.
So we can apply to \eqref{e:again economic model over F_p} the Lau equivalence $\fL$ (see Proposition~\ref{p:Lau equivalence}) and get a $\BZ$-graded 
DG ring ind-scheme $C_{n,\BF_p}^\oplus$, where 
\begin{equation} \label{e:economic model for reduction of sheared sR_n oplus}
C_{n,\BF_p}^\oplus:=\cone ((\hat W_{\BF_p}^{(F^n)})^\oplus\overset{d}\longrightarrow W_{n,\BF_p}^\oplus ).
\end{equation} 

We claim that $C_{n,\BF_p}^\oplus$ is a model for $\SR_{n,\BF_p}^\oplus$. To justify this claim, we will construct\footnote{The construction is the $m=0$ case of \S\ref{sss:economic models for hat sR_n}.} a quasi-isomorphism
\begin{equation}   \label{e:B_n epi C_n over F_p}
B_{n,\BF_p}^\oplus\epi C_{n,\BF_p}^\oplus ,
\end{equation}
where $B_{n,\BF_p}^\oplus$ is the model from \S\ref{ss:applying Lau equivalence to B_{n,F_p}}. 

\subsubsection{Constructing \eqref{e:B_n epi C_n over F_p}} 
$C_{n,\BF_p}$ is the quotient of the DG ring $B_{n,\BF_p}$ from \S\ref{sss:V on B_n in characteristic p} by the acyclic ideal $\cone (I'_{n,\BF_p}\to V^n(W_{\BF_p}))$, where
\[
I'_{n,\BF_p}=\{ (x,y)\in W^2_{\BF_p}\,|\, x=V^ny\};
\]
note that $I'_{n,\BF_p}\subset I_{n,\BF_p}$, where $I_n$ is given by \eqref{e:description of I_n}. The map $B_{n,\BF_p}\epi C_{n,\BF_p}$ commutes with $F$ and $V$. So we get \eqref{e:B_n epi C_n over F_p}

\subsection{The ring stack $\Cone (\hat W^{(F^n)}\longrightarrow  W_n)$ for arbitrary $p$}
\begin{lem}   \label{l:the naive ring stack for any p}
For any prime $p$ and any $n\in\BN$, the ring stack
$\Cone (\hat W^{(F^n)}\longrightarrow  W_n)$ is canonically isomorphic to $\Cone (\SW\overset{p^n\bu}\longrightarrow\SW)$, where $\bu\in W(\BZ_p)^\times$ is as in \S\ref{sss:the Witt vector u}.
\end{lem}

Note that $p^n\bu\in\SW (\BZ_p)$ because $p\bu=\bp\in\SW (\BZ_p)$.

\begin{proof}
Using Lemma~\ref{l:kernel and cokernel} and mimicking the proof of \cite[Prop.~3.5.1]{Prismatization}, one gets a canonical isomorphism
$\Cone (\hat W^{(F^n)}\to W_n)\iso\Cone (\SW\overset{F^n\hV^n}\longrightarrow\SW)$. By \S\ref{sss:warning to myself}(i-ii) and the inclusion \eqref{e:doesn't matter}, 
$\Cone (\SW\overset{F^n\hV^n}\longrightarrow\SW)$ is canonically isomorphic to $\Cone (\SW\overset{p^n\bu}\longrightarrow\SW)$.
\end{proof}

\section{A class of models for $\SR_n^\oplus$}      \label{s:A class of models}
\subsection{Ind-affineness of certain morphisms}  \label{ss:ind-affineness}
A morphism of stacks $\sX\to\sY$ is said to be ind-affine if for any $p$-nilpotent affine scheme $S$ and any morphism $S\to\sY$ the stack $\sX\times_\sY S$ is an ind-affine ind-scheme (i.e., it can be represented as an inductive limit of a directed family of affine schemes with respect to closed immersions).

\begin{prop}   \label{p:ind-affineness}
The diagonal morphism $\SR_n\to\SR_n\times\SR_n$ is ind-affine. 
\end{prop}

The proof will be given in \S\ref{sss:proof of ind-affineness}.

\begin{cor}   \label{c:ind-affineness}
The morphism of stacks $\SR_n^\oplus\to \SR_n [u,u^{-1}]\times \SR_n [t,t^{-1}]$ from \S\ref{ss:Defining hat sR_n}  is ind-affine. 
\end{cor}

\begin{proof}
Combine Propositions~\ref{p:positive/nonpositive degrees} and \ref{p:ind-affineness}.
\end{proof}

To prove Proposition~\ref{p:ind-affineness}, we need the following lemma.

\begin{lem}  \label{l:ind-affineness}
Let $R$ be a ring. Let $H$ be a commutative affine group $R$-scheme such that the $R$-module $H^0(H,\cO_H)$ is projective.
Let $\cF$ be an $H^*$-torsor\footnote{In other words, $\cF$ is an fpqc sheaf on the category of $R$-schemes which is a torsor over the sheaf of sections of $H^*$.}, where $H^*$ is the Cartier dual.
Then $\cF$ is representable by an ind-affine ind-scheme over $R$.
\end{lem}

\begin{rem}    \label{r:representability of torsors}
$H$-torsors are representable by affine schemes: this is a well known consequence of the theory of flat descent.
\end{rem}

\begin{proof}[Proof of Lemma~\ref{l:ind-affineness}]
Interpreting $H^*$ as the automorphism group of the trivial extension
$$0\to\BG_m\to H\oplus\BG_m\to  H\to 0$$
and using the theory of flat descent, we see that $\cF$ defines an extension 
\begin{equation}   \label{e:extension by G_m}
0\to\BG_m\to\tilde H\to  H\to 0,
\end{equation}
where $\tilde H$ is an affine $R$-scheme. Then $\cF$ is the sheaf of splittings of \eqref{e:extension by G_m}.

Let $M_+:=H^0(H, L)$, $M_-:=H^0(H,  L^{-1})$, where $L$ is the line bundle  on $H$ corresponding to the
$\BG_m$-torsor $\tilde H\to H$. A splitting of \eqref{e:extension by G_m} can be viewed as a pair 
\[
(s_+,s_-), \quad s_+\in M_+, \; s_-\in M_-
\]
satisfying certain equations (one of them is $s_+s_-=1)$. So it remains to show that the functor
\[
\{R\mbox{-algebras}\}\to\{\mbox{Sets}\}, \quad \tilde R\mapsto\tilde R\otimes_R (M_+\oplus M_-)
\]
is an ind-affine ind-scheme over $R$. To see this, represent $M_+\oplus M_-$ as a direct summand of a free $R$-module;
this is possible because the $R$-module $H^0(H,\cO_H)$ is assumed to be projective.
\end{proof}

\subsubsection{Proof of Proposition~\ref{p:ind-affineness}} \label{sss:proof of ind-affineness}
Let $S$ be a $p$-nilpotent affine scheme equipped with a morphism $f:S\to\SR_n\times\SR_n$. 
Let $X:=S\times_{\SR_n\times\SR_n}\SR_n$, where the map $\SR_n\to\SR_n\times\SR_n$ is the diagonal. Then $X$ can be described as follows.

Let $g:=f_1-f_2$, where $f_1,f_2:S\to\SR_n$ are the components of~$f$. By \S\ref{s:tilde A_n},
$$\SR_n=\Cone (\tilde A_n^{-1}\to \tilde A_n^0),$$ 
where $A_n^{-1}$ and $A_n^0$ are additively isomorphic to $W\oplus\hat W$.
So $g:S\to\SR_n$ is given given by a $\tilde A_n^{-1}$-torsor $T\to S$ and a  $\tilde A_n^{-1}$-equivariant map $h:T\to \tilde A_n^0$. In these terms, $X=h^{-1}(0)$.

\emph{A priori}, $T$ and $X$ are fpqc sheaves. We have to prove that $X$ is an ind-affine ind-scheme. It suffices to show that $T$ is an ind-affine ind-scheme.
This follows from Lemma~\ref{l:ind-affineness} and Remark~\ref{r:representability of torsors} because the group ind-scheme $\tilde A_n^{-1}$ is isomorphic to $W\oplus\hat W$, and $\hat W$ is Cartier dual to $W$.
\qed

\subsection{A class of models for $\SR_n^\oplus$ and $\SR_{n,\BF_p}^\oplus$}   \label{ss:2economic models}
\subsubsection{}  \label{sss:tautological construction}
Suppose we have a quasi-ideal pair
\begin{equation}    \label{e:I to C}
I\overset{d}\longrightarrow C
\end{equation}
with $\Cone (I\overset{d}\longrightarrow C)=\SR_n$ such that $C$ and $I$ are ind-schemes over $\Spf\BZ_p$. (E.g., formula~\eqref{e:The economic model} provides such a pair by Lemma~\ref{l:the quotient is an ind-scheme}; this pair depends on the choice of a number $m\ge\delta_p$.)
Using \eqref{e:I to C} as an input datum, we are going to construct a model for $\SR_n^\oplus$ in a rather tautological way.

Let $\tilde C$  be the fpqc sheaf of graded rings defined by the pullback diagram
\begin{equation} \label{e:defining tilde C}
\xymatrix{
\tilde C\ar@{->>}[r] \ar[d] & \SR_n^\oplus\ar[d]\\
C[u,u^{-1}]\times C[t,t^{-1}]\ar@{->>}[r] & \SR_n[u,u^{-1}]\times \SR_n[t,t^{-1}]
}
\end{equation}
in which the lower horizontal arrow comes from the epimorphism $C\epi\SR_n$ and the right vertical arrow was constructed in \S\ref{ss:Defining hat sR_n}.
By Corollary~\ref{c:ind-affineness}, $\tilde C$ is an ind-scheme.
The two horizontal arrows of \eqref{e:defining tilde C} have the same fibers over $0$, namely $I[u,u^{-1}]\times I[t,t^{-1}]$. Thus we have constructed a quasi-ideal pair
\begin{equation}   \label{e:the quasi-ideal in tilde C}
I[u,u^{-1}]\times I[t,t^{-1}]\overset{d}\longrightarrow \tilde C
\end{equation}
whose Cone equals $\SR_n^\oplus$. Moreover, $\tilde C$ is a $\BZ$-graded ring ind-scheme, and the map $d$ from \eqref{e:the quasi-ideal in tilde C} is compatible with the gradings (assuming that $\deg t=-1$, $\deg u=1$).

By construction, we get a homomorphism of quasi-ideal pairs
\begin{equation}  \label{e:triangular diagram}
\xymatrix{
I[u,u^{-1}]\times I[t,t^{-1}] \ar[r] \ar[d]_{\id} & \tilde C\ar[d]\\
I[u,u^{-1}]\times I[t,t^{-1}]\ar[r] & C[u,u^{-1}]\times C[t,t^{-1}]
}
\end{equation}
such that the corresponding morphism of Cones is the right vertical arrow of \eqref{e:defining tilde C}.

From diagram \eqref{e:triangular diagram} one gets the following two complexes of $\BZ$-graded commutative group ind-schemes
\begin{equation} \label{e:complex1}
0\to I[t,t^{-1}]\to \tilde C\to C[u, u^{-1}]\to 0, 
\end{equation} 
\begin{equation} \label{e:complex2}
0\to I[u, u^{-1}]\to \tilde C\to C[t,t^{-1}]\to 0.
\end{equation}

\begin{prop}  \label{p:exactness statements}
The complex \eqref{e:complex1} is exact in positive degrees, and \eqref{e:complex2} is exact in non-positive degrees.
Thus we have exact sequences
\begin{equation}  \label{e:exact in positive degrees}
0\to  t^{-1} I[t^{-1}]\to \tilde C_{> 0}\to uC[u]\to 0,
\end{equation}
\begin{equation}  \label{e:exact in negative degrees}
0\to I[u^{-1}]\to \tilde C_{\le 0}\to C[t]\to 0.
\end{equation}
where $\tilde C_{> 0}$ (resp.~$\tilde C_{\le 0}$) is the positively (resp.~non-positively) graded part of $\tilde C$. 
\end{prop}

\begin{proof}
Consider the diagram
\begin{equation}  \label{e:coming from triangular diagram}
\xymatrix{
I[u,u^{-1}]\times I[t,t^{-1}] \ar[r] \ar[d] & \tilde C\ar[d]\\
I[u,u^{-1}]\ar[r] & C[u,u^{-1}]
}
\end{equation}
coming from  \eqref{e:triangular diagram}. The $\Cone$ of the upper row of \eqref{e:coming from triangular diagram} is $\SR_n^\oplus$,
and the $\Cone$ of the lower row is $\SR_n [u,u^{-1}]$. By Proposition~\ref{p:positive/nonpositive degrees}, the map $\SR_n^\oplus\to \SR_n [u,u^{-1}]$ is an isomorphism in positive degrees.
So the complex corresponding to the bicomplex \eqref{e:coming from triangular diagram} is acyclic in positive degrees. This means that \eqref{e:complex1} is acyclic in positive degrees.
The statement about \eqref{e:complex2} is proved similarly.
\end{proof}

\subsubsection{A model for $\SR_{n,\BF_p}^\oplus$}    \label{sss:2tautological construction} 
Of course, the statements from \S\ref{sss:tautological construction} and Proposition~\ref{p:exactness statements} remain valid if one works over $\Spec\BF_p$ (instead of $\Spf\BZ_p$), i.e., if the input datum is a model for $\SR_{n,\BF_p}$; 
then the output is a model for $\SR_{n,\BF_p}^\oplus$. 

Suppose that the input datum is the quasi-ideal pair
$$\hat W_{\BF_p}^{(F^n)}\to  W_{n,\BF_p}$$
from \S\ref{sss:I_{n,m}/I'_{n,m} in some cases}(ii). Define $\tilde C$ by the pullback diagram 
\begin{equation}   \label{e:the pullback diagram}
\xymatrix{
\tilde C\ar@{->>}[r] \ar[d] & \SR_{n,\BF_p}^\oplus\ar[d]\\
W_{n,\BF_p} [u,u^{-1}]\times W_{n,\BF_p} [t,t^{-1}]\ar@{->>}[r] & \SR_{n,\BF_p}[u,u^{-1}]\times \SR_{n,\BF_p} [t,t^{-1}]
}
\end{equation}
similar to \eqref{e:defining tilde C}. Our next goal is to give an explicit description of $\tilde C$ (see \S\ref{sss:description of tilde C} below).

Note that $F,V$ act on $\cone (\hat W_{\BF_p}^{(F^n)}\to  W_{n,\BF_p})$ \emph{on the nose.} So we get the model $$\cone ((\hat W_{\BF_p}^{(F^n)})^\oplus\to  W^\oplus_{n,\BF_p})$$ for $\SR_{n,\BF_p}^\oplus$.
Moreover, we have a commutative diagram
\begin{equation}  \label{e:the commutative diagram}
\xymatrix{
(\hat W_{\BF_p}^{(F^n)})^\oplus\ar[d]\ar[r]&W^\oplus_{n,\BF_p}\ar@{->>}[r] \ar[d] & \SR_{n,\BF_p}^\oplus\ar[d]\\
\hat W_{\BF_p}^{(F^n)}[u^{\pm 1}]\times\hat W_{\BF_p}^{(F^n)}[t^{\pm 1}]\ar[r]&W_{n,\BF_p} [u^{\pm 1}]\times W_{n,\BF_p} [t^{\pm 1}]\ar@{->>}[r] & \SR_{n,\BF_p}[u^{\pm 1}]\times \SR_{n,\BF_p} [t^{\pm 1}]
}
\end{equation}
whose rows are fiber sequences. Using the right square of \eqref{e:the commutative diagram} and the pullback diagram \eqref{e:the pullback diagram}, we get a homomorphism $W^\oplus_{n,\BF_p}\to\tilde C$. Thus we get a commutative diagram
\begin{equation}  \label{e:the push-out diagram}
\xymatrix{
(\hat W_{\BF_p}^{(F^n)})^\oplus \ar[r] \ar[d] & W^\oplus_{n,\BF_p}\ar@{->>}[r]\ar[d]& \SR_{n,\BF_p}^\oplus\ar[d]^{\id}\\
\hat W_{\BF_p}^{(F^n)} [u,u^{-1}]\times \hat W_{\BF_p}^{(F^n)} [t,t^{-1}] \ar[r] & \tilde C\ar@{->>}[r] & \SR_{n,\BF_p}^\oplus
}
\end{equation}
whose rows are fiber sequences. From this diagram we get the following description of $\tilde C$.

\subsubsection{Description of $\tilde C$}   \label{sss:description of tilde C}
The left square of \eqref{e:the push-out diagram} is a pushout diagram. The middle vertical arrow of \eqref{e:the push-out diagram} is injective (because the left vertical arrow is).
The ring structure on $\tilde C$ is such that the action of $W^\oplus_{n,\BF_p}$ on $\hat W_{\BF_p}^{(F^n)} [u,u^{-1}]\times \hat W_{\BF_p}^{(F^n)} [t,t^{-1}]$ via the homomorphism $$W^\oplus_{n,\BF_p}\mono\tilde C$$ is equal to the action via the homomorphism $W^\oplus_{n,\BF_p}\to W_{n,\BF_p} [u,u^{-1}]\times W_{n,\BF_p} [t,t^{-1}]$.

\subsubsection{Remark}  \label{sss:rather non-economic}
The lower row of \eqref{e:the push-out diagram} provides a model
\[
\cone (\hat W_{\BF_p}^{(F^n)} [u,u^{-1}]\times \hat W_{\BF_p}^{(F^n)} [t,t^{-1}]\to\tilde C )
\]
for $\SR_{n,\BF_p}^\oplus$; this model is rather non-economic, but it has the following advantage:
\emph{in the case $p>2$ it lifts to a model for $\SR_n^\oplus$.} This follows from \S\ref{sss:I_{n,m}/I'_{n,m} in some cases}(i) combined with \S\ref{sss:tautological construction}.

\appendix

\section{A description of $\SW (R)$ for a class of $p$-nilpotent rings $R$}  \label{s:SW (R)}
Recall that $\SW:=W\times_Q Q^{\perf}$, where $Q:=W/\hat W$.
If $R\in\pNilp$ is such that $R_{\red}$ is perfect then $Q(R)=W (R)/\hat W (R)$ by \S\ref{sss:definition of Q}, so $\SW (R)$ is rather explicit.
We will give an even more explicit description of $\SW (R)$ if $R\in\pNilp$ is \emph{weakly semiperfect} in the sense of \S\ref{ss:weakly semperfect} below (this condition  is stronger than perfectness of  $R_{\red}$);
see Corollary~\ref{c:M-L} and \S\ref{sss:why explicit}.

\subsection{Weakly semiperfect $\BF_p$-algebras}   \label{ss:weakly semperfect}
Recall that an $\BF_p$-algebra $A$ is said to be \emph{semiperfect} if the Frobenius homomorphism $\Fr_A: A\to A$ is surjective.
We say that an $\BF_p$-algebra is \emph{weakly semiperfect} if it has the equivalent properties from the following lemma.

\begin{lem}     \label{l:weakly semiperfect}
The following properties of an $\BF_p$-algebra $A$ are equivalent:

(i) there exists $n\in\BN$ such that $\im\Fr_A^n=\im\Fr_A^{n+1}$;

(ii) there exists $n\in\BN$ such that $A/(\Ker\Fr_A^n)$ is semiperfect;

(iii) there exists an ideal $I\subset A$ such that $A/I$ is semiperfect and $I\subset\Ker\Fr_A^n$ for some $n$.
\end{lem}

\begin{proof}
Both (i) and (ii) are equivalent to the following property:
\[
\exists n\, \forall a\, \exists a' \mbox{ such that } \Fr_A^n (a-\Fr_A (a'))=0.
\]
It is clear that (ii)$\Leftrightarrow$(iii).
\end{proof}

\subsection{Some lemmas}

\begin{lem}     \label{l:killed by a power of F}
Let $R\in\pNilp$. Let $I\subset R$ be an ideal whose image in $R/pR$ is killed by a power of Frobenius.
Then $W(I)$ is killed by a power of $F$.
\end{lem}

\begin{proof}     
The lemma clearly holds if $R$ is an $\BF_p$-algebra. So it remains to prove the lemma if $I\subset pR$ and $pI=0$. In this case $W(I)$ is killed by $F$ (e.g., because $I$ is an $\BF_p$-algebra with zero Frobenius endomorphism).
\end{proof}

\begin{cor}   \label{c:killed by a power of F}
Let $R$ and $I$ be as in Lemma~\ref{l:killed by a power of F}. Then the pro-objects corresponding to the projective systems
\[
\ldots \overset{F}\longrightarrow W(R)\overset{F}\longrightarrow W(R), 
\]
\[
 \ldots \overset{F}\longrightarrow \hat W(R)\overset{F}\longrightarrow \hat W(R)
\]
do not change (up to canonical isomorphism) if $R$ is replaced by $R/I$. \qed
\end{cor}

Applying Corollary~\ref{c:killed by a power of F} for $I=pR$, one gets the following statement.

\begin{cor}   \label{c:R flat}
For any $R\in\pNilp$, one has canonical isomorphisms
\[
\hat W^{\perf}(R)\iso \hat W^{\perf}(R/pR), \quad W^{\perf}(R)\iso W^{\perf}(R/pR)=W(R^\flat ), 
\]
where $R^\flat:=(R/pR)^{\perf}$. \qed
\end{cor}

\subsection{The case where $R/pR$ is weakly semiperfect}
By \S\ref{sss:definition of Q}, for any $R\in\pNilp$ such that $R_{\red}$ is perfect, one has an exact sequence
\begin{equation}   \label{e:Q(R) as a quotient}
0\to\hat W(R)\to W(R)\to Q(R)\to 0.
\end{equation}

\begin{lem}   \label{l:M-L}
If $R/pR$ is weakly semiperfect then \eqref{e:Q(R) as a quotient} induces an exact sequence
\begin{equation}
0\to\hat W^{\perf}(R)\to W^{\perf}(R)\to Q^{\perf}(R)\to 0.
\end{equation}
\end{lem}

\begin{proof}  
The problem is to check surjectivity of the map $W^{\perf}(R)\to Q^{\perf}(R)$. It suffices to show that the projective system
\[
 \ldots \overset{F}\longrightarrow \hat W(R)\overset{F}\longrightarrow \hat W(R)
\]
satisfies the Mittag-Leffler condition. Weak semiperfectness of $A:=R/pR$ means that $\im\Fr_A^n=\im\Fr_A^{n+1}$ for some $n$, so the projective system
\[
 \ldots \overset{F}\longrightarrow \hat W(A)\overset{F}\longrightarrow \hat W(A)
\]
satisfies the Mittag-Leffler condition. It remains to use Corollary~\ref{c:killed by a power of F}.
\end{proof}

\begin{cor}   \label{c:M-L}
If $R/pR$ is weakly semiperfect then the map 
$$(W\times_Q W^{\perf})(R)\to (W\times_Q Q^{\perf})(R)=\SW (R)$$ 
is surjective. So the sequence 
\begin{equation}    \label{e:Akhil's sequence3 (R)}
0\to \hat W^{\perf}(R)\to (W\times_Q W^{\perf})(R)\to\SW (R)\to 0. 
\end{equation}
induced by \eqref{e:Akhil's sequence3} is exact.   \qed
\end{cor}

\subsubsection{Remarks}   \label{sss:why explicit}
(i) As explained at the end of \S\ref{sss:Some exact sequences}, $W\times_Q W^{\perf}$ canonically identifies with the semidirect product $W^{\perf}\ltimes\hat W$.
By Corollary~\ref{c:R flat}, $W^{\perf}(R)\simeq W(R^\flat )$. So the description of $\SW (R)$ by the exact sequence \eqref{e:Akhil's sequence3 (R)} is quite explicit.

(ii) Recall that the map $W\times_Q W^{\perf}\to\SW$ commutes with $F$ and $\hV$ if the operators $F$ and $\hV$ on $W\times_Q W^{\perf}$ are defined as in \S\ref{sss:hV on another fiber product}.
So the description of $F,\hV:\SW\to\SW$ in terms of  \eqref{e:Akhil's sequence3 (R)} is also quite explicit.

\subsubsection{}
In  \S\ref{ss:admissible rings}-\ref{ss:semiperfect} we apply Corollary~\ref{c:M-L} to two classes of $p$-nilpotent rings.

\subsection{$\SW (R)$ if $R$ is admissible in the sense of \cite{Lau14}}  \label{ss:admissible rings}
\subsubsection{Admissible rings}    \label{sss:admissible rings}
Let $R\in\pNilp$ be admissible in the sense of \cite{Lau14}; by definition, this means that $R_{\red}$ is perfect and the nilpotent radical of $R/pR$ is killed by a power of Frobenius.
Let $I$ be the nilradical of $R$.  Applying Corollary~\ref{c:killed by a power of F} to~$I$, we get an isomorphism $W^{\perf}(R)\iso W^{\perf}(R_{\red})=W((R_{\red})^{\perf})=W(R_{\red})$. Let $s:W(R_{\red})\to W(R)$ be the composition
$W(R_{\red})\iso W^{\perf}(R)\to W(R)$; then $s$ is a splitting\footnote{In fact, $s$ is the unique splitting, see \cite[\S 1B]{Lau14}.} for the epimorphism $W(R)\epi W(R_{\red})$. So one has $W(R)=s(W(R_{\red}))\oplus W(I)$.

We have $\hat W(R)=\hat W(I)\subset W(I)$. Following \cite{Lau14}, set 
$\BW(R):=s(W(R_{\red}))\oplus W(I)$; this is a subring of $W(R)$.

\begin{prop}    \label{p:admissible rings}
If $R\in\pNilp$ is admissible then the canonical map $\SW (R)\to W(R)$ is injective, and its image equals $\BW(R)$.
\end{prop}

\begin{proof}
Admissibility implies that $R/pR$ is weakly semiperfect, so we can apply Corollary~\ref{c:M-L}. Since $W^{\perf}(R)=W(R_{\red})$, we have $(W\times_Q W^{\perf})(R)=\BW (R)$.
By Corollary~\ref{c:killed by a power of F}, $\hat W^{\perf}(R)=\hat W^{\perf}(R_{\red})=0$.
\end{proof}

\subsection{$\SW (R)$ if $R$ is a semiperfect $\BF_p$-algebra}  \label{ss:semiperfect}

\subsubsection{Notation}
Let $R$ be a semiperfect $\BF_p$-algebra. Then $W(R)=W(R^{\perf})/W(J)$, where $J:=\Ker (R^{\perf}\epi R)$.
The ideal $W(J)\subset W(R^{\perf})$ is the projective limit of the diagram
\[
\ldots \overset{F}\longrightarrow W^{(F^2)}(R)\overset{F}\longrightarrow W^{(F)}(R).
\]
Define $\hat W_{\ttop}(J)\subset W(J)$ to be the projective limit of the diagram
\begin{equation}   \label{e:the diagram with p-complete terms}
\ldots \overset{F}\longrightarrow \hat W^{(F^2)}(R)\overset{F}\longrightarrow \hat W^{(F)}(R).
\end{equation}
Then $\hat W_{\ttop}(J)$ is an ideal in $W(R^{\perf})$.

One can also describe $\hat W_{\ttop}(J)$ as the set of those Witt vectors over $J$ whose components converge to $0$ with respect to the natural topology of $R^{\perf}$ (i.e., the projective limit topology).

\subsubsection{The formula for $\SW (R)$}
The exact sequence~\eqref{e:Akhil's sequence3 (R)} implies that
\begin{equation}   \label{e:SW(semiperfect)}
\SW (R)=W(R^{\perf})/\hat W_{\ttop}(J).
\end{equation}
One can also deduce \eqref{e:SW(semiperfect)} directly from the formulas  
$$\SW:=W\times_Q Q^{\perf} \quad \mbox{and} \quad Q(R)=W (R)/\hat W (R).$$

In \cite[\S 3.5]{BMVZ} 
one can find more material about $\SW (R)$, where $R$ is a semiperfect $\BF_p$-algebra; e.g., there is a very explicit description if $R=\BF_p[x^{\frac{1}{p^\infty}}]/(x)$.

\subsubsection{Derived $p$-completeness of $\SW (R)$}
By Lemma~\ref{l:derived completeness of SW} and Remark~\ref{r:naive derived p-completeness is good}, $\SW (R)$ is derived $p$-complete for \emph{all} $R\in\pNilp$.
If $R$ is a semiperfect $\BF_p$-algebra then derived $p$-completeness of $\SW (R)$ can also be proved as follows. $W(R^{\perf})$ is clearly $p$-complete (in the sense of \S\ref{sss:hat W of p-complete}).
Each term of \eqref{e:the diagram with p-complete terms} is killed by a power of $p$ and therefore $p$-complete. So $\hat W_{\ttop}(J)$ is $p$-complete. Since $W(R^{\perf})$ and $\hat W_{\ttop}(J)$ are $p$-complete, the r.h.s. of \eqref{e:SW(semiperfect)} is derived $p$-complete.

\subsubsection{Remark (E.~Lau)}
If $R$ is semiperfect but not perfect then $\SW (R)$ is \emph{not $p$-complete}. Indeed, \eqref{e:SW(semiperfect)} implies that $\SW (R)/(p^n)=W_n(R)$, so the projective limit of the rings $\SW (R)/(p^n)$ equals $W(R)=W(R)^{\perf}/W(J)$. But $\hat W_{\ttop}(J)\ne W(J)$ unless $J=0$, i.e., unless $R$ is perfect.

\section{The notion of derived $p$-completeness}  \label{s:derived completeness}
In \S\ref{ss:derived completeness for Z-modules} we recall the notion of derived $p$-completeness for $\BZ$-modules.
In \S\ref{ss:derived completeness for sheaves} we recall the notion of derived $p$-completeness for \emph{sheaves} of $\BZ$-modules.

\subsection{Derived $p$-completeness for $\BZ$-modules} \label{ss:derived completeness for Z-modules}
\subsubsection{}   \label{sss:derived completeness for Z-modules}
Given a ring $R$ and an ideal $I\subset R$, there is a notion of derived completeness (with respect to $I$) for $R$-modules, see \cite[Tag 091S]{Sta}.
If $R=\BZ$ and $I=p\BZ$ one gets the notion of \emph{derived $p$-complete $\BZ$-module.} The class of derived $p$-complete $\BZ$-modules is larger and ``better'' than
the class of $p$-complete ones in the sense of \S\ref{sss:hat W of p-complete}. E.g., if $f:A\to A'$ is a homomorphism of $\BZ$-modules and $A,A'$ are derived $p$-complete then so are $\Ker f$ and $\Coker f$ (see \cite[Tag 091U]{Sta}); on the other hand, if $A,A'$ are $p$-complete then $\Coker f$ is $p$-complete if and only if $f(A)$ is closed for the $p$-adic topology of $A'$.
For more details, see \cite{Sta} (starting with \cite[Tag 091N]{Sta}) and references therein.

Recall that derived $p$-completeness of a $\BZ$-module $A$ is equivalent to each of the following conditions:

(i) the map $A\to\limfromn (A\overset{L}\otimes (\BZ/p^n\BZ))$ is an isomorphism;

(ii) the derived projective limit of the diagram 
\[
\ldots \overset{p}\longrightarrow A\overset{p}\longrightarrow A.
\]
is zero.

The above derived projective limit equals $R\Hom (\BZ [p^{-1}],A)$. So (ii) can be rewritten as
\begin{equation}  \label{e:RHom=0}
R\Hom (\BZ [p^{-1}],A)=0.
\end{equation}

Once we choose an exact sequence
\[
0\to\BZ^{(\BN)}\overset{f}\longrightarrow\BZ^{(\BN)}\to\BZ [p^{-1}]\to 0, 
\]
where $\BZ^{(\BN)}:=\BZ\oplus\BZ\oplus\ldots$, we can rewrite condition \eqref{e:RHom=0} 
as follows:

\smallskip

(iii) the map $A^{\BN}\overset{f^*}\longrightarrow A^{\BN}$ induced by $f$ is an isomorphism (here $A^{\BN}:=A\times A\times\ldots$).

\subsubsection{Remarks}  \label{sss:2derived completeness for Z-modules}
(i) A $\BZ$-module killed by a power of $p$ is derived $p$-complete: indeed, derived $p$-completeness is equivalent to \eqref{e:RHom=0}.

(ii) A projective limit of derived $p$-complete $\BZ$-modules is derived $p$-complete: this is clear from \S\ref{sss:derived completeness for Z-modules}(iii).

(iii) $A$ is $p$-complete in the sense of \S\ref{sss:hat W of p-complete} if and only if the map from $A$ to the projective limit of $A/p^nA$ is an isomorphism. So  $p$-completeness implies derived $p$-completeness by the above Remarks (i)-(ii).

\subsection{Derived $p$-completeness for sheaves of $\BZ$-modules}  \label{ss:derived completeness for sheaves}
Now let $A$ be a \emph{sheaf} of $\BZ$-modules on a site $\cC$. 

\subsubsection{}  \label{sss:derived completeness for sheaves}
According to \cite[Tag 0995]{Sta}, $A$ is said to be derived $p$-complete if it satisfies the condition from \S\ref{sss:derived completeness for Z-modules}(i), which is equivalent to the one from \S\ref{sss:derived completeness for Z-modules}(ii). The latter can be rewritten as
\begin{equation}  \label{e:RHHom=0}
R\HHom (\underline{\BZ} [p^{-1}],A)=0,
\end{equation}
where $\underline{\BZ} [p^{-1}]$ is the constant sheaf with fiber $\BZ [p^{-1}]$ and $\HHom$ denotes the sheaf of homomorphisms.

It is clear that if two terms of an exact sequence of sheaves of $\BZ$-modules are derived $p$-complete then so is the third one.

\begin{lem}   \label{l:repleteness}
Assume that the site $\cC$ satisfies the following condition: a countable product of exact sequences of sheaves of $\BZ$-modules on $\cC$ is exact.
Then

(a) derived $p$-completeness of $A$ is equivalent to the condition from \S\ref{sss:derived completeness for Z-modules}(iii);

(b) $A$ is derived $p$-complete if and only if the $\BZ$-module $A(c)=H^0(c,A)$ is derived $p$-complete for each $c\in\cC$.
\end{lem}

\begin{proof}
By assumption, $\Ext^i(\underline{\BZ}^{(\BN)},A)=0$ for $i>0$. So \eqref{e:RHHom=0} is equivalent to the condition from \S\ref{sss:derived completeness for Z-modules}(iii). This proves statement (a). It implies (b)
because a Cartesian product of sheaves is just their product in the sense of presheaves.
\end{proof}

\subsubsection{Remark}  \label{sss:repleteness}
By \cite[Prop.~3.1.9]{BS}), the condition from Lemma~\ref{l:repleteness} holds if $\cC$ is replete in the sense of \cite[Def.~3.1.1]{BS}.

\begin{lem}   \label{l:Picard stacks}
Let $A$ be a sheaf of $\BZ$-modules. Let $S_n$ denote $\Cone (A\overset{p^n}\longrightarrow A)$ viewed as a Picard stack. If $A$ is derived $p$-complete then the map $A\to\limfromn S_n$ is an isomorphism.
\end{lem}

\begin{proof}   
$S_n$ is the stack of extensions of $p^{-n}\underline{\BZ}/\underline{\BZ}$ by $A$. 
So $\limfromn S_n$ is the stack of extensions of $\underline{\BZ}[p^{-1}]/\underline{\BZ}$ by $A$. It remains to use \eqref{e:RHHom=0}.
\end{proof}

\bibliographystyle{alpha}

\end{document}